\documentclass{amsart}

\usepackage{graphicx}
\usepackage{amssymb,latexsym}
\usepackage{a4wide}
\usepackage{psfrag}
\usepackage[all]{xy}
\xyoption{2cell}
\UseTwocells

\let\cal\mathcal
\def\AA{{\cal A}}
\def\BB{{\cal B}}
\def\CC{{\cal C}}

\def\FF{{\cal F}}
\def\GG{{\cal G}}

\def\II{{\cal I}}

\def\LL{{\cal L}}

\def\OO{{\cal O}}
\def\PP{{\cal P}}

\def\SS{{\cal S}}
\def\TT{{\cal T}}
\def\UU{{\cal U}}

\let\blb\mathbb
 
\def\bD{{\blb D}}

\def\bP{{\blb P}}

\def\bZ{{\blb Z}}

\def\bN{{\blb N}}

\def\bZ{{\blb Z}}

\let\frak\mathfrak
 
\def\fg{\frak{g}}

\def\aa{\frak{a}}
\def\bb{\frak{b}}

\def\ff{\frak{f}}

\def\Cohom{\operatorname{Cohom}}
\def\Comod{\operatorname{Comod}}

\def\Mod{\operatorname{Mod}}

\def\mod{\operatorname{mod}}
\def\modc{\mod^{\rm{cfp}}}

\def\nilp{\operatorname{nilp}}

\def\coh{\mathop{\text{\upshape{coh}}}}

\def\rad{\operatorname {rad}}
\def\PC{\operatorname {PC}}

\def\Rep{\operatorname {Rep}}

\def\repc{\rep^{\rm{cfp}}}
\def\rep{\operatorname{rep}}

\def\Ext{\operatorname {Ext}}

\def\Hom{\operatorname {Hom}}

\def\End{\operatorname {End}}

\def\im{\operatorname {im}}

\def\coker{\operatorname {coker}}
\def\ker{\operatorname {ker}}

\def\End{\operatorname {End}}

\DeclareMathOperator{\Ind}{Ind}

\DeclareMathOperator{\Mor}{Mor}
\DeclareMathOperator{\colim}{colim}

\DeclareMathOperator{\soc}{soc}

\DeclareMathOperator{\ind}{ind}

\newcommand\Db{D^{b}}

\renewcommand\t{\tau}

\newtheorem{lemma}{Lemma}[section]
\newtheorem{proposition}[lemma]{Proposition}
\newtheorem{theorem}[lemma]{Theorem}
\newtheorem{corollary}[lemma]{Corollary}

\theoremstyle{definition}

\newtheorem{example}[lemma]{Example}
\newtheorem{definition}[lemma]{Definition}

{

}

\theoremstyle{remark}

\newtheorem{remark}[lemma]{Remark}

\newdimen\uboxsep \uboxsep=1ex
\def\uboxn#1{\vtop to 0pt{\hrule height 0pt depth 0pt\vskip\uboxsep
\hbox to 0pt{\hss #1\hss}\vss}}

\def\uboxs#1{\vbox to 0pt{\vss\hbox to 0pt{\hss #1\hss}
\vskip\uboxsep\hrule height 0pt depth 0pt}}

\def\Ob{\operatorname{Ob}}

\def\PC{\operatorname{PC}}

\def\Cat{\mathfrak{Cat}}

\def\CAT{\mathfrak{Cat}}

\newcommand\exa{\nopagebreak \begin{center}\smallskip \nopagebreak               \begin{minipage}[t]{6cm}\sloppy}
\newcommand\exb{\end{minipage}\kern 1cm\begin{minipage}[t]{8cm}\sloppy}
\newcommand\exc{\end{minipage}\kern -3cm \smallskip\end{center}}

\def\Ja{J^{\overline{\aa}}}
\def\Jb{J^{\overline{\bb}}}
\def\oa{\overline{\aa}}
\def\ob{\overline{\bb}}

\title{Hereditary uniserial categories with Serre Duality}
\author{Adam-Christiaan van Roosmalen}
\address{Adam-Christiaan van Roosmalen\\Fakult\"at f\"ur Mathematik\\ Universit\"at Bielefeld\\D-33501 Bielefeld\\Germany}\email{vroosmal@math.uni-bielefeld.de}


\begin{document}

\bibliographystyle{amsplain}

\begin{abstract}
An abelian Krull-Schmidt category is said to be uniserial if the isomorphism classes of subobjects of a given indecomposable object form a linearly ordered poset.  In this paper, we classify the hereditary uniserial categories with Serre duality.  They fall into two types: the first type is given by the representations of the quiver $A_n$ with linear orientation (and infinite variants thereof), the second type by tubes (and an infinite variant).  These last categories give a new class of hereditary categories with Serre duality, called \emph{big tubes}.
\end{abstract}

\maketitle

\setcounter{tocdepth}{1}
\tableofcontents

\section{Introduction}

We will fix an algebraically closed field $k$, and will only consider $k$-linear categories.  We will say that a Hom-finite abelian category $\AA$ is uniserial if for every indecomposable $X \in \Ob \AA$ the subobjects of $X$ are linearly ordered by inclusion.  The uniserial hereditary length categories with only finitely many (nonisomorphic) simple objects have been classified in \cite{AmdalRingdal68} (see also \cite{ChenKrause09, CuadraGomezTorrecillas04, Gabriel73}).

\begin{theorem}\label{theorem:Length}
Let $\AA$ be a hereditary uniserial length category with finitely many isomorphism classes of simple objects, then $\AA$ is equivalent to either
\begin{itemize}
\item the category $\rep_k A_n$ of finite dimensional representations of the quiver $A_n$ with linear orientation, or
\item the category $\nilp_k \tilde{A}_n$ of finite dimensional nilpotent representations of the quiver $\tilde{A}_n$ with cyclic orientation.
\end{itemize}
\end{theorem}

Categories of the second type will be referred to as \emph{tubes}.  We wish to replace the condition on the length of the objects by the existence of a Serre functor.  Note that the two classes of categories mentioned above do have Serre duality, so that the existence of a Serre functor is a (strictly) weaker condition.  Our main result is the following (Theorem \ref{theorem:Classification} in the text).

\begin{theorem}\label{theorem:Introduction}
Let $\AA$ be an essentially small $k$-linear uniserial hereditary category with Serre duality.  Then $\AA$ is equivalent to one of the following
\begin{enumerate}
\item the category $\repc \LL$ of finitely presented and cofinitely presented representations of $\LL$ where $\LL$ is a locally discrete linearly ordered poset, either without minimal or maximal elements, or with both a minimal and a maximal element, or
\item a (big) tube.
\end{enumerate}
\end{theorem}

We refer to \S\ref{section:BigTubes} for definitions.  Categories of the first type have been introduced in \cite{vanRoosmalen06} (see also \cite{Ringel02}); they are directed and generalize categories of the form $\rep A_n$ with $n \in \bN$.  Every object in $\repc \LL$ is finitely presented (it is the cokernel of a map between finitely generated projectives in $\Rep \LL$) and is cofinitely presented (it is the kernel of a map between finitely cogenerated injectives in $\Rep \LL$).  When $\LL$ has a minimal and a maximal element, then $\repc \LL \cong \rep \LL$ and they are equivalent to category of finitely presented representations of a thread quiver $\xymatrix{\cdot \ar@{..>}[r]^\PP & \cdot}$ for a linearly ordered poset $\PP$, possibly empty (see \cite{BergVanRoosmalen09} for the definition of a thread quiver and its representations.).

A big tube is an infinite generalization of a tube (A definition is given in \S\ref{section:BigTubes}).  It is not a length category and has infinitely many isomorphism classes of indecomposable simple objects.  Every indecomposable object in a big tube lies in a subcategory of the form $\nilp_k \tilde{A}_n$ and thus (roughly speaking) we may see a big tube as a union of its subtubes.  This approach (to consider a big tube as a filtered 2-colimit of tubes) will be taken in \S\ref{section:2Colimit} to finish the proof of Theorem \ref{theorem:Introduction}.

%To this end, we will embed a (small) Hom-finite length category in its category of Ind-objects.  Such categories will be Grothendieck categories of finite type, and these are described using coalgebras.  In \S\ref{section:Grothendieck}, we will give a brief recount of the theory of coalgebras (and dually pseudocompact algebras) which we will need to complete the proof of the classification in theorem \ref{theorem:Introduction}.

We do not know of any mention of big tubes in the literature, and thus believe these to be a new type of hereditary categories with Serre duality.  These categories are used, for example, in the construction of Hall algebras or cluster categories.

Our main motivation for introducing and studying big tubes comes from the following.  The category $\rep_k Q$ of finite dimensional representations of a tame quiver $Q$ has tubes as subcategories.  More precisely, every regular module lies in a tube and the embedding $\nilp_k \tilde{A}_n \to \rep_k Q$ maps Auslander-Reiten sequences to Auslander-Reiten sequences (so that the embedding maps the Auslander-Reiten quiver of $\nilp_k \tilde{A}_n$ to a component of the Auslander-Reiten quiver of $\rep_k Q$; it is of the form $\bZ A_\infty / \langle \t^{n+1} \rangle$).

By replacing the tame quiver $Q$ by a ``nice tame infinite variant'' one obtains a category $\AA$ of representations which contains a big tube as a full subcategory, such that the Auslander-Reiten translations in $\AA$ and in the big tube coincide.  Here, a ``nice tame infinite variant'' of $Q$ is given by a certain thread quiver.  We refer to \S\ref{subsection:Threads} for an example of a big tube in the category of representations of the thread quiver 

Likewise, in the category of coherent sheaves on weighted projective lines (see \cite{GeigleLenzing87} or \cite{Lenzing07}), the simple objects are contained in tubes.  The aforementioned example can be considered as (being derived equivalent to) an example of a weighted projective line with an infinite weight (Remark \ref{remark:InfiniteWeight}).

The proof of the classification consists of three steps.  Let $\AA$ be any essentially small $k$-linear uniserial hereditary category with Serre duality.  The first step (Proposition \ref{proposition:Uniserial}) is to prove some consequences of Serre duality, most importantly that $\AA$ has ``enough simples'' in the sense that every indecomposable object has a simple socle and a simple top.  We will then take a cofinite subset of isomorphism classes of simple objects and consider the perpendicular subcategory.  The second step of the proof is to say that this category is a hereditary uniserial length category with finitely many simple objects (Proposition \ref{proposition:FiniteShape}).

It then follows that $\AA$ is a filtered 2-colimit of such length subcategories $\AA_i$.  To prove that $\AA$ is equivalent to a category in Theorem \ref{theorem:Introduction}, we will take an appropiate such category $\BB$ and similarly write it as a filtered 2-colimit of length subcategories $\BB_i$.  Finding an equivalence between $\AA$ and $\BB$ is then the same as finding a consistent set of equivalences between the length subcategories $\AA_i$ of $\AA$ and the corresponding ones of $\BB$.  To find such equivalences, we will embed the categories $\AA_i$ and the categories $\BB_i$ in their Ind-closures.  Such categories are locally finite Grothendieck categories of finite type and functors between them are described using coalgebras or (dually) pseudocompact algebras.  We will use this structure to define a consistent set of functors from the subcategories $\AA_i$ to the subcategories $\BB_i$, inducing an equivalence between $\AA$ and $\BB$, completing the classification.   Relevant definitions and theorems about Grothendieck categories in are given in \S\ref{section:Grothendieck}.

%The article consists of three parts.  In the first part we define and discuss a big loop and a big tube.  In the second part, we prove that all hereditary uniserial categories with Serre duality are a union (or a filtered 2-colimit) of hereditary uniserial length categories.  In the third part, we combine this last property with Theorem \ref{theorem:Length} to complete the classification in Theorem \ref{theorem:Introduction}.

\textbf{Acknowledgments} The author wishes to thank Michel Van den Bergh for meaningful discussions, and wishes to thank Jan \v S\v tov\'\i\v cek and Joost Vercruysse for many useful comments on an early draft.  The author also gratefully acknowledges the financial and administrative support of the Hausdorff Center for Mathematics in Bonn.

\section{Preliminaries}

\subsection{Notations and conventions}

We will assume all categories are $k$-linear for an algebraically closed field $k$.  A category $\CC$ is called \emph{Hom-finite} if $\dim_k \Hom_\CC (A,B) < \infty$ for all $A,B \in \Ob \CC$, and an abelian category $\CC$ is called \emph{Ext-finite} if and only if $\dim_k \Ext^{i}_\CC (A,B) < \infty$ for all $A,B \in \Ob \CC$ and all $i \geq 0$.  An abelian Ext-finite category is thus automatically also Hom-finite.  An abelian category will be called \emph{semi-simple} if $\Ext^1(-,-) = 0$ and \emph{hereditary} if $\Ext^2(-,-) = 0$.

We will also choose a Grothendieck universe $\UU$ and assume all our categories are $\UU$-categories, i.e. every Hom-sets in the category is an element of $\UU$.  A category is called \emph{$\UU$-small} (or just small) if the object-set is also an element of $\UU$ and it is called \emph{essentially $\UU$-small} (or essentially small) if it is equivalent to a $\UU$-small category.

Following \cite[Theorem A]{ReVdB02} we will say that an Ext-finite hereditary category $\AA$ has \emph{Serre duality} \cite{BondalKapranov89} if and only if $\AA$ has almost split sequences and there is a one-one correspondence between the indecomposable projective objects $P$ and the indecomposable injective objects $I$, such that the simple top of $P$ is isomorphic to the simple socle of $I$.  The Auslander-Reiten translate in $\AA$ will be denoted by $\tau$.

\subsection{Paths in Krull-Schmidt categories}

For a Krull-Schmidt category $\AA$ we will denote by $\ind \AA$ a (chosen) maximal set of nonisomorphic indecomposables of $\AA$.  For a Krull-Schmidt subcategory $\BB$ of $\AA$, we will choose $\ind \BB$ as a subset of $\ind \AA$.  We wish to advice the reader to not confuse the set $\ind \AA$ with the category $\Ind \AA$ of ind-objects mentioned below.

Let $\AA$ be a Krull-Schmidt category and $A,B \in \ind \AA$.  An \emph{unoriented path} from $A$ to $B$ is a sequence of objects $A=X_0,X_1,\ldots, X_n = B$ such that $\Hom(X_i,X_{i+1}) \not= 0$ or $\Hom(X_{i+1},X_{i}) \not= 0$ for all $0 \leq i < n$.  Similarly, an \emph{oriented path} (also abbreviated to \emph{path}) from $A$ to $B$ is a sequence of objects $A=X_0,X_1,\ldots, X_n = B$ such that $\Hom(X_i,X_{i+1}) \not= 0$ for all $0 \leq i < n$.

An abelian category $\AA$ will be called \emph{indecomposable} if and only if it is nonzero and not equivalent to the product category of two nonzero categories.  If $\AA$ is Hom-finite, and hence Krull-Schmidt, then $\AA$ is indecomposable if and only if there is an unoriented path between any two indecomposable objects of $\AA$.

\subsection{Perpendicular subcategories}

Let $\AA$ be an abelian Ext-finite hereditary category and let $\SS \subseteq \Ob \AA$.  We will denote by $\SS^\perp$ the full subcategory of $\AA$ consisting of all objects $X$ with $\Hom(\SS,X) = \Ext(\SS,X) = 0$, called the category \emph{right perpendicular to $\SS$}.  It follows from \cite[Proposition 1.1]{GeigleLenzing91} that $\SS^\perp$ is again an abelian hereditary category and that the embedding $\SS^\perp \to \AA$ is exact.  If $\SS = \{E\}$ consists of a single object $E \in \Ob \AA$, then we will also write $E^\perp$ for $\SS^\perp$.

Let $E \in \Ob \AA$ be an exceptional object (i.e. $\Ext(E,E) = 0$).  It follows from \cite[Proposition 3.2]{GeigleLenzing91} that the embedding $i:E^\perp \to \AA$ has a left adjoint $L:\AA \to E^\perp$.  If $E$ is furthermore a simple object, then the left adjoint $L:\AA \to E^\perp$ is exact (see \cite[Proposition 2.2]{GeigleLenzing91}).  Also note that $L$ maps a simple object of $\AA$ to either a simple object in $E^\perp$ or to zero.

For easy reference, we will combine these results in a proposition.

\begin{proposition}\cite{GeigleLenzing91}\label{proposition:GeigleLenzing}
Let $\AA$ be an abelian Ext-finite hereditary category.  Let $\SS \subseteq \Ob \AA$ be a finite set of simple and exceptional objects.  The category $\SS^\perp$ is abelian and hereditary, the embedding $\SS^\perp \to \AA$ is exact and has an exact left adjoint.
\end{proposition}

\subsection{2-colimits}

In this article, we will sometimes see a category as a union of some suitable small subcategories.  In other words, some categories of consideration will be 2-colimits of smaller categories (all 2-colimits in this article can be seen as unions of full subcategories).  We will repeat some definitions and results from \cite{Waschkies04} (see also \cite{Borceux94, Borceux94b, Kelly05}).  We will work in the strict 2-category $\Cat$ of small categories, thus:
\begin{itemize}
\item the 0-cells are given by small categories,
\item the 1-cells are functors,
\item the 2-cells are natural transformations.
\end{itemize}
Composition of 1-cells is denoted by $\circ$.  Following \cite{MacLane71, Waschkies04} we will write $\circ$ for vertical composition of 2-cells and $\bullet$ for horizontal composition.

%We will define a 2-functor $\aa: \PP \to \Cat$ to be a functor from a (small) category $\PP$ to the underlying 1-category of $\Cat$.  More precisely:
\begin{definition}
Let $\PP$ be a small 1-category.  A \emph{2-functor (with strict identities)} $\aa: \PP \to \Cat$ is given by the following data:
\begin{enumerate}
\item a 0-cell $\aa(i)$ of $\Cat$ for every $i \in \Ob \PP$,
\item a 1-cell $\aa(s): \aa(i) \to \aa(j)$ of $\Cat$ for every morphism $s:i \to j$ in $\PP$ and $\aa(1_i) = 1_{\aa(i)}$ for all $i \in \Ob \PP$,
\item a natural equivalence $\Phi(s,t): \aa(t \circ s) \stackrel{\sim}{\rightarrow} \aa(t) \circ \aa(s)$ for all composable morphisms $s,t \in \Mor \PP$,
\end{enumerate}
satisfying the following condition: for three composable morphisms $u,t,s \in \Mor \PP$, we have the following commutative diagram
$$\xymatrix@C=80pt{
\aa(u \circ t \circ s) \ar[r]^{\Phi(t \circ s, u)} \ar[d]_{\Phi(s, u \circ t)} & \aa(u)\aa(t \circ s) \ar[d]^{1_{\aa(u) \bullet \Phi(s,t)}} \\
\aa(u \circ t)\aa(s) \ar[r]^{\Phi(t,u) \bullet 1_{\aa(s)}} & \aa(u) \aa(t) \aa(s)
}$$
A 2-functor is called \emph{strict} if $\Phi(s,t) = 1$, thus $\aa(t \circ s) = \aa(t) \circ \aa(s)$ for all composable $s,t \in \Mor \PP$
\end{definition}

\begin{remark}
A strict 2-functor is just a functor from $\II$ to the underlying 1-category of $\Cat$.
\end{remark}

\begin{example}
For every object $\CC$ of $\Cat$, there is a 2-functor $\CC: \PP \to \Cat$ sending every object of $\PP$ to $\CC$ and sending every morphism of $\CC$ to the identity on $\CC$.
\end{example}

\begin{definition}
Let $\aa, \bb: \PP \to \Cat$ be two 2-functors.  A 2-natural transformation $\ff: \aa \to \bb$ between 2 diagrams consists of the following data:
\begin{enumerate}
\item a 1-cell $\ff_i: \aa(i) \to \bb(i)$ of $\Cat$ for every $i \in \Ob \PP$, and
\item a natural equivalence $\theta^\ff_s: \bb(s) \circ \ff_i \to \ff_j \circ \aa(s)$ for every morphism $s:i \to j$ in $\PP$.
\end{enumerate}
such that for any two composable morphisms $s:i \to j$, $t: j \to k$ in $\PP$, we have the following commutative diagram
$$\xymatrix@C=80pt{
\bb(t \circ s) \circ \ff_i \ar[r]^{\Phi^{\bb}(s,t) \bullet 1_{\ff_i}} \ar[dd]_{\theta^\ff_{t \circ s}}& \bb(t)\circ \bb(s) \circ \ff_i \ar[d]^{1_{\bb(t)} \bullet \theta^\ff_s}\\
& \bb(t) \circ \ff_i \circ \aa(s) \ar[d]^{\theta^\ff_t \bullet 1_{\aa(s)}} \\
\ff_k \circ \aa(t \circ s) \ar[r]^{\ff_k \bullet \Phi^{\aa}(s,t)} & \ff_k \circ \aa(t) \circ \aa(s)
}$$
\end{definition}

\begin{definition}
Let $\aa, \bb: \II \to \Cat$ be two 2-functors and $\ff, \fg: \aa \to \bb$ be 2-natural transformations.  A \emph{modification} $\Lambda: \ff \to \fg$ consists in giving a 2-cell $\Lambda_i: \ff_i \to \fg_i$ for all objects $i \in \II$ such that for all $s:i \to j$ in $\II$ the following diagram commutes
$$\xymatrix{
\bb(s) \circ \ff_i \ar[r]^{\theta^\ff_s} \ar[d]_{1_{\bb(s)} \bullet \Lambda_i}&\ff_j \circ \aa(s) \ar[d]^{\Lambda_j \bullet 1_{\aa(s)}} \\
\bb(s) \circ \fg_i \ar[r]^{\theta^\fg_s} & \fg_j \circ \aa(s)}$$
\end{definition}

\begin{definition}
The diagrams, 2-natural transformations, and modifications form a (strict) 2-category called $2\FF(\II,\Cat)$.
\end{definition}

We can now give the definition of a 2-colimit.

\begin{definition}
Let $\aa: \II \to \Cat$ be a 2-functor.  We say $\aa$ admits a 2-colimit if and only if there exist
\begin{enumerate}
\item a category $2\colim \aa$, and
\item a 2-natural transformation $\sigma: \aa \to 2\colim \aa$,
\end{enumerate}
such that for every category $\CC$ the functor
$$(- \circ \sigma): \Hom_{\Cat}(2\colim \aa, \CC) \to \Hom_{2\FF}(\aa, \CC)$$
is an equivalence of categories.
\end{definition}

The 2-category $\Cat$ of all small categories has 2-colimits (\cite[Theorem A.3.4]{Waschkies04}).

\begin{theorem}
Let $\II$ be a small category and $\aa: \II \to \Cat$ a 2-functor.  Then $\aa$ admits a 2-colimit.
\end{theorem}

The following result is \cite[Proposition A.5.5]{Waschkies04}

\begin{proposition}\label{proposition:AbelianColimit}
Let $\aa:\PP \to \Cat$ be a 2-functor where $\PP$ is a small filtered category.  Suppose that $\aa(i)$ is an additive (abelian) category for any $i \in \Ob \PP$ and that $\aa(s)$ is an additive (exact) functor for every morphism $s \in \Mor \PP$.  Then $2\colim \aa$ is an additive (abelian) category and the natural functors $\sigma_i:\aa(i) \to 2\colim \aa$ are additive (exact).
\end{proposition}

The next result (\cite[Proposition A.3.6]{Waschkies04}) will be applicable to all 2-colimits we will consider.

\begin{proposition}\label{proposition:FilteredColimit}
Let $\PP$ be a small filtered category such that between two given objects there is at most one morphism, and let $\aa: \PP \to \Cat$ be a 2-functor such that every functor $\aa(s)$ is fully faithful ($s \in \Mor \PP$).  Any object $X \in 2\colim \aa$ is isomorphic to an object of the form $\sigma_i(X')$ where $X' \in \aa(i)$.  For any $i,j \in \Ob \PP$, $X \in \Ob \aa(i)$, and $Y \in \Ob \aa(j)$, we have that
$$\Hom_{2\colim \aa}(\sigma_i X, \sigma_j Y) \cong \Hom_{\aa(k)}(\aa(s)X,\aa(t)Y)$$
where $k \in \Ob \PP$ such that there are morphisms $s:i \to k$ and $t:j \to k$.  The above isomorphism is given by $\varphi_k: \aa(k) \to 2\colim \aa$.
\end{proposition}
\section{Grothendieck categories and tubes}\label{section:Grothendieck}

In the proof of our main theorem, we wil be interested in functors between categories without injectives.  In order to handle such functors better, we will embed such a category $\AA$ first in its category of Ind-objects $\Ind \AA$ defined below.  Such a category will be a Grothendieck category of finite type if $\AA$ is an (essentially small) Hom-finite length category.

In this section, we recall some relevant definitions and results.  Our aim is Corollary \ref{corollary:WhenEquivalent} which describes the functors we will be interested in.  We will only use these results when $\AA$ is a tube (as the other categories we will consider have enough injectives); the category of Ind-objects of a tube is briefly described in \S\ref{subsection:Tubes}.

\subsection{Locally finite Grothendieck categories}

An abelian category is called a \emph{Grothendieck category} if it has a generator and exact direct limits.  It is well-known that a Grothendieck category has injective envelopes \cite[Theorem II.2]{Gabriel62} and an injective cogenerator.

Let $\AA$ be an essentially small abelian category.  We denote by $\Ind \AA$ the full subcategory of $\Mod \AA$ consisting of all left exact contravariant functors $\AA \to \Mod k$.  It has been shown in \cite{Gabriel62} that $\Ind \AA$ is a Grothendieck category.  Every object $A \in \Ind \AA$ can be written as a formal small filtered colimit in $\AA$ (thus an object of $\Ind \AA$ is given by a functor from a small filtered category to $\AA$) and the Hom-sets may be computed by
$$\Hom_{\Ind \AA} (\underset{\longrightarrow}{\lim\nolimits}{}_i A_i, \underset{\longrightarrow}{\lim}{}_j B_j) = \underset{\longleftarrow}{\lim}{}_i {\underset{\longrightarrow}{\lim\nolimits}}{}_j \Hom_\AA(A_i,B_j)$$

If $\AA$ and $\BB$ are essentially small abelian categories and $F:\AA \to \BB$ is a functor, then $F$ lifts to a functor $\overline{F}: \Ind \AA \to \Ind \BB$ as follows (\cite{KashiwaraShapiro06})
$$\overline{F}(\underset{\longrightarrow}{\lim}{}_i A_i) = \underset{\longrightarrow}{\lim}{}_i F(A_i).$$
The action on the Hom-spaces is the obvious one.  If $F$ is faithful, fully faithful, left exact, or right exact, then the same holds for $\overline{F}$.  Furthermore, it follows easily from the definition that a left or right adjoint functor  $L,R: \BB \to \AA$ of $F$ lifts to a left or right adjoint functor $\overline{L},\overline{R}: \BB \to \AA$, repsectively.

We have the following.

\begin{proposition}\label{proposition:ToInjectives}
Let $F: \AA \to \BB$ be a functor with an exact left adjoint between two essentially small categories, then $\overline{F}:\Ind \AA \to \Ind \BB$ maps injective objects to injective objects.
\end{proposition}

\begin{proof}
We know that $\overline{F}$ has an exact left adjoint $\overline{L}$.  For any injective $I \in \Ind \AA$ we have $\Hom(-,\overline{F}I) \cong \Hom(\overline{L}-,I)$.  Since this last functor is exact, we know that $\overline{F}I$ is injective.
\end{proof}

%More generally, this can be formulated as the following proposition.
%
%\begin{proposition}
%There is a 2-functor from the 2-category of small abelian categories to the 2-category of Grothendieck categories.
%\end{proposition}

A Grothendieck category is called \emph{locally finite} if it has a small generating set consisting of objects of finite length.  Thus if $\AA$ is a $\Hom$-finite abelian length category, then $\Ind \AA$ is a locally finite Grothendieck category.  One can recover $\AA$ from $\Ind \AA$ as the full subcategory of finite length objects, or as the full subcategory of compact objects.  Recall that an object $A \in \Ind \AA$ is called \emph{compact} if and only if the functor $\Hom(A,-): \Ind \AA \to \Mod k$ commutes with arbitrary direct sums.

For locally finite categories, we have the following result (\cite[Theorem IV.2]{Gabriel62}, see also \cite{Matlis58}) concerning injective objects.

\begin{theorem}\label{theorem:Matlis}
Let $\BB$ be a locally finite category.  Every injective object is a direct sum of indecomposable injective objects, and all direct sums of injectives are injective.  Moreover, this decomposition is essentially unique up to permutation of the direct summands.
\end{theorem}

An object $B$ in a locally finite Grothendieck category $\BB$ is said to be \emph{quasi-finite} if and only if $\dim \Hom_\BB (X,B) < \infty$ for all $X$ of finite length.  A Grothendieck category is said to be of \emph{finite type} if and only if it is locally finite and $\dim_k \Hom(A,B) < \infty$ for all $A,B$ of finite length; thus all objects of finite length are quasi-finite.

Let $X \in \Ob \BB$.  It has been shown in \cite[Proposition 1.3]{Takeuchi77} that the functor $X \otimes_k -: \Mod k \to \BB$ has a left adjoint if and only if $X$ is quasi-finite.  This left adjoint will be denoted by $\Cohom(X,-): \BB \to \Mod k$.  %Using the adjunction, we find that $\Cohom(X,Y) \cong \underset{\longrightarrow}{\lim}{}_i \Hom(Y_i,X)^*$ where $(-)^*$ is the vector space dual and $Y_i$ ranges over all finite length subobjects of $Y$.
The adjunction implies the following universal property: for each $X,Y \in \Ob \BB$ where $X$ is quasi-finite, there is a map $\theta: Y \to \Cohom(X,Y) \otimes_k X$ in $\BB$ such that each map $Y \to V \otimes_k X$ factors uniquely through $\theta$.

Note that $\Cohom(X,Y)^* = \Hom_k(\Cohom(X,Y),k) \cong \Hom_\BB(Y,X \otimes_k k) \cong \Hom_\BB(Y,X)$.

\subsection{Coalgebras and pseudocompact algebras}

It is known \cite{Gabriel62} that every locally finite Grothendieck category is dual to a category of pseudocompact modules over a pseudocompact ring.  If the Grothendieck category is of finite type, then it is equivalent to the category of comodules over a coalgebra (\cite{Takeuchi77}).  We will recall some definitions about pseudocompact algebras and coalgebras, following the exposition in \cite{VandenBergh10} closely.  We refer to \cite{Gabriel62} for more information on pseudocompact rings, and to \cite{BrzezinskiWisbauwer03} for information on coalgebras and comodules.

The category $\Mod k$ of $k$-vector spaces with the usual tensor product is a monodial category.  A \emph{$k$-coalgebra} is a coalgebra object in this category, i.e. a triple $(C, \Delta, \epsilon)$ where $C \in \Mod k$, $\Delta: C \to C \otimes_k C$, and $\epsilon: C \to k$ such that the following diagrams commute
$$\xymatrix@C+25pt{C \ar[r]^\Delta \ar[d]_{\Delta}& C \otimes_k C \ar[d]^{1_C \otimes_k \Delta} & C\ar@{=}[rd] \ar[r]^{\Delta} \ar[d]_{\Delta}&  C \otimes_k C \ar[d]^{\epsilon \otimes_k 1_C} \\
C \otimes_k C \ar[r]_{\Delta \otimes_k 1_C}& C \otimes_k C \otimes_k C & C \otimes_k C \ar[r]_{1_C \otimes_k \epsilon} & C}$$
The maps $\Delta$ and $\epsilon$ are called the \emph{comultiplication} and the \emph{counit}, respectively.  We will write $C$ for the coalgebra $(C,\Delta,\epsilon)$, leaving the maps $\Delta$ and $\epsilon$ understood.  Let $C$ and $D$ be coalgebras.  We define \emph{left $C$-comodules}, \emph{right $C$-comodules}, and \emph{$(C,D)$-bicomodules} in the usual way (see for example \cite{BrzezinskiWisbauwer03}).  The category of left or right $C$-comodules will be denoted by $C-\Comod$ or $\Comod-C$, respectively.

\begin{remark}
When $X$ is a quasi-finite object in a locally finite Grothendieck category, then the universal map $\theta: X \to \Cohom(X,X) \otimes X$ induces a coalgebra structure on $\Cohom(X,X)$.  We refer to \cite{Takeuchi77} for details.
\end{remark}

A topological vector space is called \emph{pseudocompact} if it is complete and it has a basis of open subsets which are subspaces of finite codimension.  The category of pseudocompact vector spaces will be denoted by $\PC(k)$, the morphisms are continuous $k$-linear maps.  Note that every pseudocompact vector space is Hausdorff.  The topology of a finite dimensional pseudocompact vector space is necessarily the discrete topology, and conversely every such topological vector space is pseudocompact.

The category $\PC(k)$ is dual to the category $\Mod k$; the dualities are given by
\begin{eqnarray*}
\bD: \Mod k \to \PC (k)^\circ &:& V \mapsto \Hom_{k}(V,k) \\
\bD: \PC (k) \to \Mod k ^\circ &:& V \mapsto \Hom_{\PC(k)}(V,k)
\end{eqnarray*}
where the topology on $\Hom_k(V,k)$ is generated by the kernels of $\bD V \to \bD W$ whenever $W$ is a finite dimensional subspace of $V$.

We can use the duality $\bD$ to give $\PC(k)$ the structure of a monoidal category as follows: for all $V,W \in \PC(k)$ we have
$$V \otimes W = \bD (\bD W \otimes_k \bD V).$$
Note that there are isomorphisms $k \otimes V \cong V \cong V \otimes k$ in $\PC(k)$, natural in $V$.

\begin{remark}
The tensor product on $\PC(k)$ will be denoted by $- \otimes -$ while the tensor product on $\Mod k$ will be denoted $- \otimes_k -$.
\end{remark}

\begin{remark}
For two pseudocompact vector spaces $\bD V, \bD W$ (thus $V,W \in \Mod k$), we have $\bD V \otimes \bD W \cong \bD (W \otimes_k V)$.  This is the completion of $\bD V \otimes_k \bD W$ (taking the tensor product in $\Mod k$, i.e. as vector spaces).
\end{remark}

A \emph{pseudocompact $k$-algebra} is an algebra object in the category $\PC(k)$, i.e. it is a triple $(A,m,e)$ where $A \in \PC(k)$, $m: A \otimes A \to A$, and $e: k \to A$ such that the following diagrams commute
$$\xymatrix@C+25pt{A \otimes A \otimes A \ar[r]^{m \otimes 1} \ar[d]_{m \otimes 1_A}& A \otimes A \ar[d]^{m} & A \ar@{=}[rd] \ar[r]^{1_A \otimes e} \ar[d]_{e \otimes 1_A}&  A \otimes A \ar[d]^{m} \\
A \otimes A \ar[r]_{m}& A & A \otimes A \ar[r]_{m} & A}$$
We will always write $A$ for the pseudocompact algebra $(A,m,e)$.  Let $A$ be a pseudocompact algebra.  \emph{Pseudocompact modules} over pseudocompact algebras are defined using similar diagrams.  We will denote the category of left pseudocompact $A$-modules by $\PC(A)$.

We see that pseudocompact algebras are are dual to coalgebras, i.e. if $C$ is a coalgebra, then $\bD C$ is a pseudocompact algebra and vice versa.

\begin{remark}
If $A$ is a pseudocompact algebra, then $A$ is also an algebra using the inclusion $A \otimes_k A \to A \otimes A$.  When the multiplication $A \otimes_k A \to A$ is continuous, then it lifts to a map $A \otimes A \to A$.  Hence a pseudocompact $k$-vector space with a continuous multiplication is the same as a pseudocompact $k$-algebra. \end{remark}

\begin{example}
The algebra $k[[x]]$ is a pseudocompact algebra where the topology is generated by the subspaces $(x^n) \triangleleft k[[x]]$.  This pseudocompact algebra is dual to the divided power coalgebra $k[x]$, i.e. the comultiplication is given by $\Delta(x^n) = \sum_{i+j=n} x^i \otimes_k x^j$.
\end{example}

\begin{remark}\label{remark:RingIsNotAlgebra}
A pseudocompact ring (\cite{Gabriel62}) which is also a $k$-algebra is not necessarily a pseudocompact $k$-algebra.  For example, take $K = k(x)$.  With the discrete topology, $K$ is a pseudocompact ring but there is no topology such that $K$ is a pseudocompact $k$-algebra.  Indeed, $\dim_k K = \aleph_0$ and such objects cannot occur in the essential image of $\bD: \Mod k \to \PC(k)$.  We note however that $K$ is a pseudocompact $K$-algebra with the discrete topology.
\end{remark}

\begin{proposition}
Let $C$ be a coalgebra and $A = \bD C$ be the dual pseudocompact $k$-algebra.  The categories $\Comod-C$ and $\PC(A)$ are dual.
\end{proposition}

\begin{proof}
The duality is induced by $\bD$, mapping a right $C$-comodule $(M,\Delta:M \to M \otimes_k C)$ to a left pseudocompact $A$-module $(\bD M, \bD \Delta: \bD M \otimes \bD (M \otimes_k C) \cong \bD C \otimes \bD M)$.
\end{proof}

The following theorem is \cite[Theorem 5.1]{Takeuchi77}.

\begin{theorem}\label{Theorem:FiniteTypeClassification}
Every Grothendieck category $\AA$ of finite type is equivalent to the category $\operatorname{Comod-}C$ for some coalgebra $C$.  The coalgebra $C$ is given by $C \cong \Cohom(I,I)$ where $I$ is a quasi-finite injective cogenerator for $\AA$.
\end{theorem}

Expressing the previous theorem in terms of pseudocompact algebras gives us the following.
\begin{corollary}\label{Corollary:FiniteTypeClassification}
Every Grothendieck category $\AA$ of finite type is dual to the category $\PC(A)$ for some pseudocompact $k$-algebra $A$.  The pseudocompact algebra $A$ is given by $A = \End I$ where $I$ is a quasi-finite injective cogenerator for $\AA$.
\end{corollary}

\begin{remark}
In the statement of Corollary \ref{Corollary:FiniteTypeClassification}, the topology on $A$ comes from $A \cong \bD \Cohom (I,I)$.  This topology corresponds with the usual topology on $A$, namely the topology generated by the kernels of $A \cong \Hom(I,I) \to \Hom(I_i,I)$ where $I_i$ ranges over the finite length subobjects of $I$.
\end{remark}

\begin{remark}
In \cite{Gabriel62} it has been shown that every locally finite Grothendieck category (thus not necessarily of finite type) is dual to the category $\PC(R)$ for a pseudocompact ring.  If the category $\AA$ is $k$-linear then $R$ will be a $k$-algebra, but $R$ might fail to be a pseudocompact $k$-algebra as the following example illustrates.
\end{remark}

\begin{example}
Let $K = k(x)$ as in Remark \ref{remark:RingIsNotAlgebra}.  With the discrete topology $K$ is a pseudocompact ring and hence $\PC(K)^\circ$ is a locally finite category.  It is however not of finite type over $k$ as $K$ is simple in $\PC(K)^\circ$ but $\dim_k (\End K) = \infty$.
Note that $\PC(K)^\circ$ is of finite type over $K$ and that $K$ with the discrete topology is a pseudocompact $K$-algebra (as $K$ is finite dimensional over $K$).
\end{example}

\subsection{Functors between Grothendieck categories of finite type}

Let $\AA, \BB$ be Grothendieck categories of finite type.  We will be interested in left exact functors $\AA \to \BB$ which commute with direct sums.  Since $\AA$ and $\BB$ are equivalent to categories of comodules of certain coalgebras, the description of such functors is given by a version of the Eilenberg-Watts Theorem for coalgebras.  We refer to \cite{BrzezinskiWisbauwer03} or \cite{Takeuchi77} for a proof.

\begin{theorem}\label{theorem:CoalgebraFunctors}
Let $\AA, \BB$ be Grothendieck categories of finite type.  Let $I$ and $J$ be quasi-finite injective cogenerators of $\AA$ and $\BB$, respectively.  Denote by $C$ and $D$ the $k$-coalgebras $\Cohom(I,I)$ and $\Cohom(J,J)$, respectively.  The category of all left exact functors $\AA \to \BB$ which commute with direct sums is equivalent to the category of $(C,D)$-bicomodules.
\end{theorem}

Recall (Theorem \ref{Theorem:FiniteTypeClassification}) that $\AA$ and $\BB$ are equivalent to the categories of right $\Cohom(I,I)$ comodules and right $\Cohom(J,J)$ comodules, respectively.  The above correspondance between functors and bimodules is then given by mapping a $(C,D)$-bicomodule $M$ to the functor $- \square_C M$.  Here $- \square_C M: \Comod-C \to \Comod-D$ is the cotensor product (see \cite{BrzezinskiWisbauwer03}).  Conversely, given a left exact functor $F: \AA \to \BB$ which commutes with direct sums, the associated bicomodule is $F(I)$.  Note that $F(I)$ is a left $C$-comodule by $F(I) \to F(C \otimes_k I) \cong C \otimes_k F(I)$.

Thus in order to specify a left exact functor $\AA \to \BB$ which commutes with direct sums, we may give an object $M \in \BB$ and give it a left $C$-comodule structure.  This is equivalent to giving a coalgebra morphism $\varphi: E \to C$ where $E = \Cohom(M,M)$.

Indeed, such a morphism $\varphi: E \to C$ induces a coaction $M \to E \otimes_k M \stackrel{\varphi \otimes_k 1_M}{\longrightarrow} C \otimes_k M$.  Conversely given a coaction $M \to C \otimes_k M$, the universal property of the map $M \to E \otimes_k M$ gives a map $E \to C$ which we can check to be a coalgebra morphism.

In particular, a functor $F: \AA \to \BB$ which commutes with direct sums such that $F(I) = M$ induces a coalgebra morphism $E \to C$.  We will use the following corollary of Theorem \ref{theorem:CoalgebraFunctors}.

\begin{corollary}\label{corollary:bicomodules}
Let $\AA, \BB$ be Grothendieck categories of finite type.  Let $I$ be a quasi-finite injective cogenerator of $\AA$, and let $M$ be any quasi-finite object in $\BB$.  We will denote $C = \Cohom(I,I)$ and $D' = \Cohom(M,M)$.  For every coalgebra map $\varphi: D' \to C$ there is a left exact functor (unique up to natural equivalence) $F: \AA \to \BB$ which commutes with direct sums.
\end{corollary}

Using the language of pseudocompact algebras, we find

\begin{corollary}\label{corollary:WhenEquivalent}
Let $\AA, \BB$ be Grothendieck categories of finite type.  Let $I$ be a quasi-finite injective cogenerator of $\AA$, and let $M$ be any quasi-finite object in $\BB$.  We will denote $A = \End(I)$ and $B' = \End(M)$ with the natural topologies.  For every continuous morphism $\varphi: A \to B'$ there is a left exact functor (unique up to natural equivalence) $F: \AA \to \BB$ which commutes with direct sums.
\end{corollary}

\begin{proof}
This is a rephrasing of Corollary \ref{corollary:bicomodules} obtained by replacing the coalgebra homomorphism $\Cohom(M,M) \to \Cohom(I,I)$ by the dual map $\End(I) \to \End(M)$.
\end{proof}

\subsection{Tubes}\label{subsection:Tubes}

Let $Q$ be an $\tilde{A}_n$-quiver with cyclic orientation.  The category $\AA = \nilp Q$ of finite dimensional nilpotent $k$-representations is an abelian hereditary Ext-finite uniserial length category with Serre duality.  We will call this category $\AA$ a \emph{tube}.

The Auslander-Reiten quiver of $\nilp \tilde{A}_n$ is of the form $\bZ A_\infty / \langle \t^{n+1} \rangle$ as in Figure \ref{fig:LargerTube}, where the peripheral objects correspond to the simple representations.

\begin{figure}[tb]
	\centering
		\includegraphics[totalheight=0.40\textheight]{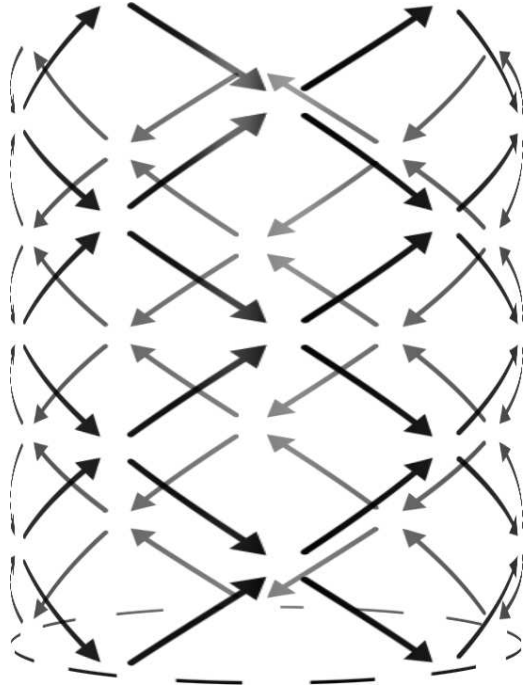}
	\caption{The Auslander-Reiten quiver of a standard tube}
	\label{fig:LargerTube}
\end{figure}

Since $\AA$ is a Hom-finite length category, $\Ind \AA$ is a Grothendieck category of finite type.  It follows from \cite{AmdalRingdal68} that the category $\Ind \AA$ has an injective cogenerator $I$ such that
$$A = \End I \cong \begin{pmatrix}
k[[x]] & xk[[x]] & \cdots & xk[[x]] & xk[[x]] \\
k[[x]] & k[[x]] & \cdots & xk[[x]] & xk[[x]]\\
\vdots & & \ddots & \vdots & \vdots \\
k[[x]] & k[[x]] &\cdots & k[[x]] &xk [[x]] \\
k[[x]] & k[[x]] &\cdots & k[[x]] &k [[x]]
\end{pmatrix}.$$
This is the completion of the path algebra of $\tilde{A}_n$ with cyclic orientation.  It follows from \cite{VanGastelVandenBergh97} that the topology is necessarily given by the product topology of all the entries in the matrix.  For each of the entries, the only pseudocompact topology is generated by the open sets $x^n k[[x]]$ for all $n \geq 1$.

If $\SS \subseteq \ind \AA$ is the set of all simple objects, then $I \cong \oplus_{S \in \SS} I(S)$ where $I(S)$ is an injective envelope of $S$.  Note that this injective cogenerator is multiplicity free and that this property determines $I$ up to isomorphism.

Another description of $\Ind \AA$ can be given using coalgebras (see \cite{{CuadraGomezTorrecillas04}}).  We can describe $\Cohom(I,I)$ as the path coalgebra of $\tilde{A}_n$, i.e. as a vector space $\Cohom(I,I)$ is generated by the paths in $\tilde{A}_n$, and the comultiplication and counit are given by
$$\Delta(a) = \sum_{pq = a} p \otimes q, \ \ \ \ \epsilon(a) = \left\{ \begin{array}{ll} 1 & \mbox{if $\operatorname{length} a = 0$} \\  0 & \mbox{if $\operatorname{length} a \not= 0$} \end{array} \right.$$
where $a$ is a path in $\tilde{A}_n$.  When $n = 0$, then $\Cohom(I,I)$ is isomorphic to the divided power coalgebra in one variable.

Note that $\End I$ is indeed dual to $\Cohom(I,I)$.
\section{Big tubes}\label{section:BigTubes}
A big tube is an infinite generalization of a tube as described in \S\ref{subsection:Tubes}.  We will define them by using an infinite version of a cyclic quiver $\tilde{A}_n$ called a big loop.  Big tubes can occur in the category of representations of certain thread quivers, an example will be provided in \S\ref{subsection:Threads}.  We start by fixing some definitions of representations of (small) preadditive categories.

\subsection{Representations of preadditive categories}\label{subsection:Representations}

Let $\aa$ be a small preadditive category.  A \emph{right $\aa$-module} is a contravariant functor from $\aa$ to $\Mod k$, the category of all vector spaces.  The category of all right $\aa$-modules is denoted by $\Mod \aa$.

With every object $A$ of $\aa$, we may associate a \emph{standard projective} $\aa(-,A)$ and a \emph{standard injective} $\aa(A,-)^*$.  It is clear that every finitely generated projective is a direct summand of a direct sum of standard projective.  If idempotents split in $\aa$ then every indecomposable projective is isomorphic to a direct sum of standard projectives, and finitely generated projectives are finite direct sums of standard projectives.  Dual notions hold for injective objects.

Let $M$ be in $\Mod(\frak{a})$. We will say that $M$ is \emph{finitely generated} if $M$ is a quotient object of a finitely generated projective object.  We say that  $M$ is \emph{finitely presented} if $M$ has a presentation
$$P\to Q\to M\to 0$$
where $P,Q$ are finitely generated projectives. It is easy to see that these notions coincides with the ordinary categorical ones.

Dually we will say that $M$ is \emph{cofinitely generated} if it is contained in a cofinitely generated injective. \emph{Cofinitely presented} is defined in a similar way.

The full subcategory of $\Mod \aa$ consisting of all objects which are finitely presented will be denoted by $\mod \aa$.  The full subcategory of $\Mod \aa$ with objects which are both finitely presented and cofinitely presented will be denoted by $\modc \aa$.

With an indecomposable object $A \in \ind \aa$, we may associate in a straightforward way the \emph{standard simple} object $S_A$ as $\aa(-,A) / \rad(-,A)$ where $\rad(-,-)$ is the usual radical.

A preadditive category $\aa$ will be called \emph{semi-hereditary} if $\mod \aa$ is abelian and hereditary.  The following theorem (\cite{vanRoosmalen06}, see also \cite{AuslanderReiten75}) characterizes semi-hereditary categories.

\begin{theorem}\label{theorem:SemiHereditary}
Let $\frak{a}$ be a small preadditive category such that any full subcategory of $\frak{a}$ with a finite number of objects is semi-hereditary. Then $\frak{a}$ is itself semi-hereditary.
\end{theorem}

Note that if $\aa$ is semi-hereditary, then so is the opposite category.  We find that $\modc \aa$ is abelian and hereditary when $\aa$ is semi-hereditary.

For a quiver $Q$, we will write $kQ$ for the associated additive $k$-linear path category, and define $\rep Q$ and $\repc Q$ to be the categories $\mod kQ$ and $\modc kQ$.  Similar conventions hold for a poset $\LL$.

\subsection{Construction of big loops and tubes}

We start with the definition of a big loop.  If $\LL$ is a linearly ordered (small) set, then we may define a (small) category $\LL^\bullet$ where the object set is given by elements of $\LL$, the morphisms by

$$\Hom_{\LL^\bullet} (i,j) = \left\{
\begin{array}{ll}
\bN & \mbox{if $i \leq j$} \\
\bN \setminus \{0\} & \mbox{if $i > j$}
\end{array} \right.$$

and where the composition is given by addition.  Note that the identity morphism in $\Hom_{\LL^\bullet}(i,i)$ is given by $0 \in \bN$.  Also, the category $\LL^\bullet$ is not $k$-linear.

The (additive) linearization $k\LL^\bullet$ of the above category may be described as follows: the objects are formal (finite) direct sums of objects of $\LL^\bullet$ and the morphisms are given by

$$\Hom_{k\LL^\bullet} (i,j) = \left\{
\begin{array}{ll}
k[x] & \mbox{if $i \leq j$} \\
xk[x] & \mbox{if $i > j$}
\end{array} \right.$$

The objects of the path completed category $\widehat{k\LL^\bullet}$ are the same as those of $k\LL^\bullet$, while the morphisms are given by

$$\Hom_{\widehat{k\LL^\bullet}} (i,j) = \left\{
\begin{array}{ll}
k[[x]] & \mbox{if $i \leq j$} \\
xk[[x]] & \mbox{if $i > j$}
\end{array} \right.$$

Composition is given by multiplication.  If $\LL$ is locally discrete without a minimum or a maximum, thus if every element in $\LL$ has a direct predecessor and successor, then the additive category $\widehat{k \LL^\bullet}$ is called a \emph{big loop}.

If $\aa$ is a big loop, then we will call the category $\modc \aa$ of finitely presented and cofinitely presented modules (cfr. \ref{subsection:Representations}) a \emph{big tube}.  Theorem \ref{theorem:SemiHereditary} implies that this category is abelian and hereditary. 

\begin{remark}
We have $k\tilde{A}_{n} \cong kA_{n+1}^\bullet$ (where the quiver $A_{n+1}$ has linear orientation) and $\nilp k\tilde{A}_{n} \cong \modc \widehat{k\tilde{A}_{n}}$.  However, with the above definitions, $k\tilde{A}_{n}$ is not a big loop and hence $\nilp \tilde{A}_n$ is not a big tube.
\end{remark}

\begin{remark}
The category $\LL^\bullet$ has been called $\LL_{cyc}$ in \cite{Drinfeld04}.  There is a natural action of the monoid $\bN$ on $\End 1_{\LL^\bullet}$ where $1_{\LL^\bullet}$ is the identity functor $\LL^\bullet \to \LL^\bullet$.  A category together with this action has been called a $\bZ_+$-category in \cite{Drinfeld04}.  Likewise, $k\LL^\bullet$ and $\widehat{k\LL^\bullet}$ can also be endowed with the structure of a $\bZ_+$-category in a natural way.
\end{remark}

\subsection{Description of a big tube}

We will now discuss the objects and morphisms occurring in such a big tube.  Since every object in $\modc \aa$ is finitely presented, it suffices to discuss the objects and morphisms of $\modc \bb$ for a well-chosen preadditive subcategory $\bb$ of $\aa$ with finitely many (indecomposable) objects.  In this case, $\modc \bb \cong \nilp \tilde{A}_n$.  Note that this implies that $\modc \aa$ is an Ext-finite abelian category (and hence Krull-Schmidt).

The indecomposable objects of $\nilp \tilde{A}_n$ are easily understood.  The simple objects are the standard simples, thus with every vertex $x$ of $\tilde{A}_n$ we associate the simple representation $S$ by $S(x) \cong k$, $S(y) \cong 0$ when $x \not= y$, and $S(\alpha)=0$ for every arrow $\alpha$ in $\tilde{A}_n$.  An indecomposable nilpotent module $M$ is uniquely determined by a simple top $T$, a simple socle $S$, and a \emph{winding number} $n \in \bN$ where $n = \dim \Hom(M,M)-1$.  Thus an indecomposable object is simple if and only if the top and the socle are isomorphic and the winding number is 0.

Likewise, in $\modc \aa$, the simple representations are given by the standard simples, and every $M \in \Ob \modc \aa$ is uniquely determined by a simple socle $S$, a simple top $T$, and a \emph{winding number} $n \in \bN$ where $n = \dim \Hom(M,M)-1$.  A module $M$ with above properties will be written as $M(s,t;n)$ where $s,t$ are indecomposable objects of $\aa$ with $S(s) \not= 0$ and $T(t) \not= 0$.  Note that the object $M(s,t;n)$ is the image of a map $\aa(-,t) \to \aa(s,-)^*$.

The category $\modc \aa$ has Serre duality.  Indeed, since it has no projectives or injectives, it suffices to check it has almost split sequences \cite[Theorem A]{ReVdB02}.  Let $M(s,t;n)$ be an indecomposable module.  It is straightforward to check that
$$\t M(s,t;n) = M(s-1,t-1;n)$$
does indeed define an Auslander-Reiten translate, where $s-1$ and $t-1$ are the direct predecessors of $s$ and $t$, respectively.

All irreducible maps are one of the following form
$$\begin{array}{rcll}
M(s,t;n) &\to& M(s+1,t;n)&\mbox{with $s+1 \not= t$} \\
M(s,t;n) &\to& M(s,t+1;n)&\mbox{with $s\not= t+1$} \\
M(s-1,s;n) &\to& M(s,s;n+1)& \\
M(s,s-1;n+1) &\to& M(s,s;n)&
\end{array}$$
where $s,t \in \Ob \aa$ and $n \in \bN$.  Every indecomposable, not of the form $M(s,s;0)$, has thus two direct successors and two direct predecessors; modules of the form $M(s,s;0)$ are simple and have only one direct successor and one direct predecessor.  The Auslander-Reiten components are thus of the form $\bZ A_\infty$ if the component has simple objects, and of the form $\bZ A_\infty^\infty$ otherwise.  We may sketch this situation as in Figure \ref{figure:BigTube} where we follow the conventions of \cite{Ringel02, Ringel02b, vanRoosmalen06} and draw Auslander-Reiten components of the form $\bZ A_\infty$ and $\bZ A_\infty^\infty$ as triangles and squares, respectively.

\begin{figure}[p]
	\centering
		\includegraphics[totalheight=.95\textheight]{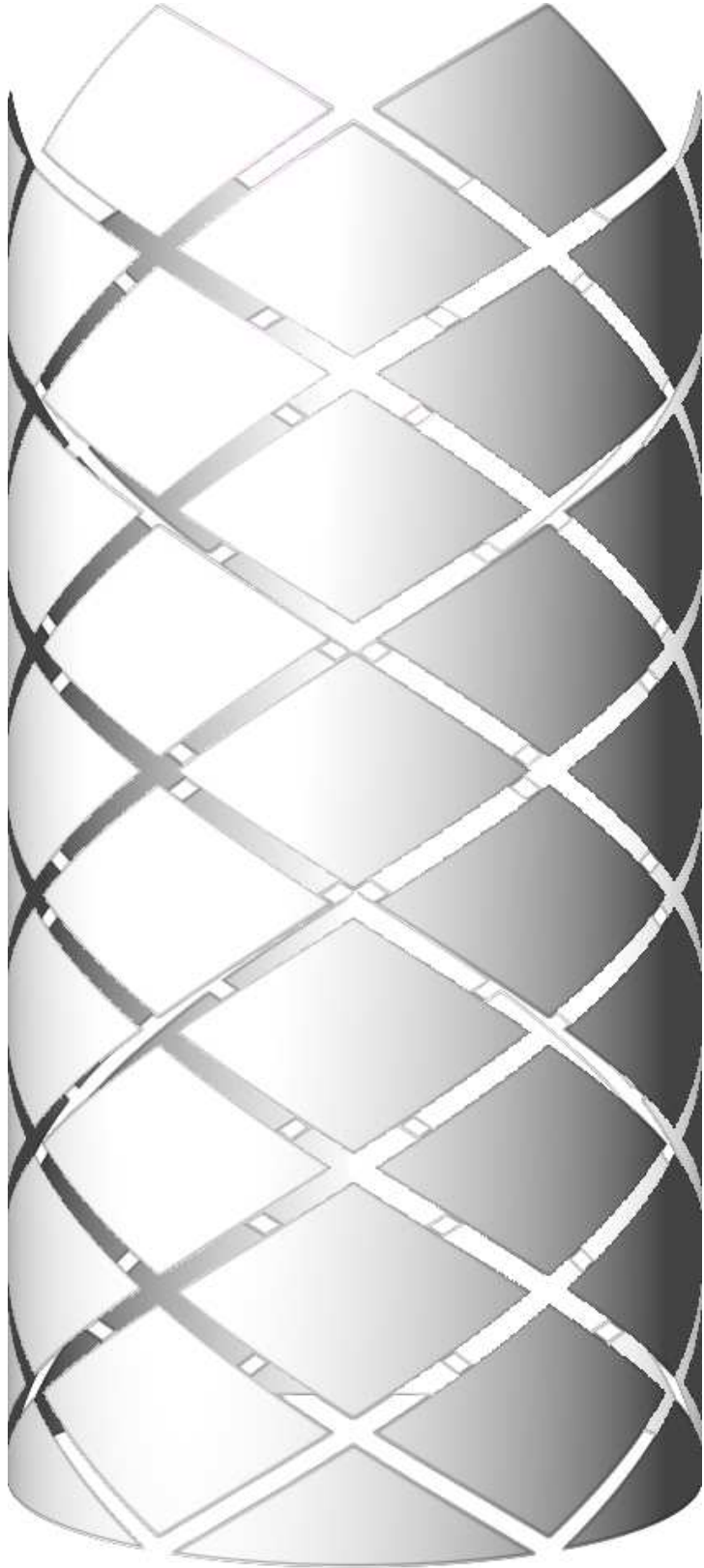}
	\caption{Sketch of the Auslander-Reiten quiver of a big tube}
	\label{figure:BigTube}
\end{figure}

\begin{remark}
Whereas an ordinary tube is a single component in the Auslander-Reiten quiver, in contrast a big tube is the union of an infinite number of components.  The objects of finite length lie in the $\bZ A_\infty$-components and the objects of infinite length lie in the $\bZ A_\infty^\infty$-components.
\end{remark}

\subsection{Big tubes in the representation of thread quivers}\label{subsection:Threads}

A big tube can also occur as a subcategory of an abelian Ext-finite hereditary category $\AA$ with Serre duality (such that the embedding commutes with the Auslander-Reiten translate), more specifically when $\AA$ is the category of finitely presented representations of a thread quiver (we refer the reader to \cite{BergVanRoosmalen09} for more information about thread quivers, although this section can be read independently).  We will discuss an example of this.

Denote by $Q_n$ ($n \in \bN$) the quiver
$$\xymatrix@R=5pt{&&&&&&&\\ 
a_0 \ar[r] \ar `u /4pt[urrrrrrr] `/4pt[rrrrrrr] [rrrrrrr] & a_1 \ar@{..}[r] & a_{n-1} \ar[r] & a_{n} \ar[r]& b_n \ar[r]& b_{n-1} \ar@{..}[r]& b_1 \ar[r]& b_0}$$
Note that $Q_0$ is the Kronecker quiver.  There is an obvious functor $f:kQ_n \to \widehat{k\tilde{A}_{2n-2}}$ mapping $a_0$ and $b_0$ to the same indecomposable in $k\tilde{A}_{2n-2}$, mapping the radical maps in the lower branch of $kQ_n$ to radical maps in $k\tilde{A}_{2n-2}$, and mapping the morphism corresponding to the upper arrow in $Q_n$ to the identity of $F(a_0) = F(b_0)$.  Thus if the upper arrow and the lower branch correspond to the morphisms $\alpha, \beta \in \Hom_{kQ_n}(a_0,b_0)$, then $F(\alpha) = 1_{F(a_0)}$ and $F(\beta)$ is a generator of the unique maximal ideal in $\Hom_{\widehat{k\tilde{A}_{2n-2}}}(F(a_0),F(a_0)) \cong k[[x]]$. 

The functor $f:kQ_n \to \widehat{k\tilde{A}_{2n-2}}$ induces a restriction functor $\modc \widehat{k\tilde{A}_{2n-2}} \to \mod kQ_n$.  The essential image of this functor is given by all objects $M$ of $\mod kQ_n$ where $M(\alpha)$ is an isomorphism and $M(\alpha)^{-1} \circ M(\beta)$ is nilpotent.  One verifies that the essential image is the tube in $\mod kQ_n$ containing the module $M$ with $M(a_0) = M(b_0) = k$, $M(\alpha) = 1$, and $M(\beta) = 0$.  If $n \not= 0$, then this is the unique nonhomogeneous tube in $\mod kQ_n$.

Denote by $f_n:kQ_n \to kQ_{n+1}$ the embedding implied by the numbering of the vertices.  This gives a filtered system and the 2-colimit (in the 2-category of small categories) is the path category $kQ$ of the thread quiver $Q=\xymatrix@1{a\ar@<-2pt>@{..>}[r]\ar@<2pt>[r]&b}$ (see \cite{BergVanRoosmalen09}).  The indecomposable objects in $kQ$ are denoted by $a_i$ and $b_i$ (for all $i \in \bN$) in a natural way.

As before, there is an obvious functor $kQ \to \widehat{k\bZ^\bullet}$ mapping $a_i$ to $i$ and $b_i$ to $-i$.  This functor induces a restriction functor $\modc \widehat{k\bZ^\bullet} \to \mod kQ$.  The description of the essential image is similar to the above description.

To see that the restriction functor commutes with Auslander-Reiten translations, we offer an alternative description of the above restriction functor.  Consider the following commutative diagram
$$\xymatrix{
kQ_0 \ar[r] \ar[d]& kQ_1 \ar[r]\ar[d]& kQ_2 \ar[r]\ar[d] & \cdots \\
\widehat{k\tilde{A}_0} \ar[r] & \widehat{k\tilde{A}_2} \ar[r] & \widehat{k\tilde{A}_4} \ar[r] & \cdots 
}$$

The 2-colimit of the upper sequence is $kQ$, of the lower sequence is $\widehat{k\bZ^\bullet}$, and the diagram induces the functor $f:kQ \to \widehat{k\bZ^\bullet}$.  We find the following sequence

$$\xymatrix{
\mod kQ_0 \ar[r] & \mod kQ_1 \ar[r]& \mod kQ_2 \ar[r] &\cdots \ar[r]& \mod kQ \\
\modc \widehat{k\tilde{A}_0} \ar[r]\ar[u] & \modc \widehat{k\tilde{A}_2} \ar[r]\ar[u] & \modc \widehat{k\tilde{A}_4} \ar[r]\ar[u] & \cdots \ar[r]& \modc \widehat{k\bZ^\bullet} \ar@{-->}[u]
}$$
where the functors in the top and bottom row are given by tensor products, and are hence exact (\cite{BergVanRoosmalen09}).  Since all objects in $\mod kQ$ are finitely presented, it is follows that the 2-colimit of the top row is $\mod kQ$.  Likewise, the 2-colimit of the bottom row is $\modc \widehat{k\bZ^\bullet}$.  The solid vertical arrows are restriction functors discussed above and hence commute with Auslander-Reiten translations.  The dashed arrow is given by the universal property of 2-colimits and is thus so that the diagram commutes (up to natural equivalence).

The Auslander-Reiten translate of an indecomposable $M(k,l;n)$ in the big tube $(k,l,n \in \bZ)$ is given by $M({k-1},{l-1};n)$.  We find an $m \in \bN$ such that these two indecomposables lie in the subcategory $\modc \widehat{k\tilde{A}_{2m-2}}$ of $\widehat{k\bZ^\bullet}$.  Since the diagram essentially commutes, this shows that the restriction $\modc \widehat{k\bZ^\bullet} \to \mod kQ$ commutes with the Auslander-Reiten translation.

\begin{remark}\label{remark:InfiniteWeight}
The previous example may be interpreted as (being derived equivalent to) a weighted projective line (\cite{GeigleLenzing87}) of weight type $(\infty)$, thus there is one point with weight $\infty$; the other points have weight one.  This notation is not unambiguous as one could replace $\bZ$ by any other (larger) linearly ordered locally discrete set.
\end{remark}
\section{Paths in hereditary categories}

In this section, we will assume the reader is familiar with the concept of a bounded derived category of an abelian category.

Let $\AA$ be an Ext-finite abelian category, so that the bounded derived category $\Db \AA$ is a Krull-Schmidt category.  A \emph{suspended path} in $\Db \AA$ is a sequence $X_0, X_1, \ldots, X_n$ of indecomposable objects such that for all $i=0, 1, \ldots, n-1$ we have $\Hom(X_i,X_{i+1}) \not= 0$ or $X_{i+1} \cong X_i[1]$.  If $\AA$ is indecomposable and not semi-simple, then it follows from \cite{Ringel05} that there is a suspended path from $X$ to $Y$ in $\Db \AA$ if and only if there is a path from $X$ to $Y$.  More generally (\cite[Lemma 5]{Ringel05}), $\AA$ is indecomposable if and only if for every $X,Y \in \ind \Db \AA$ there is an $n \in \bZ$ and a suspended path from $X$ to $Y[n]$.  In this case, $\Db \AA$ is also called a \emph{block}.

It was shown in \cite[Theorem 2.5]{HappelZacharia08} that any suspended path from $A$ to $B$ in $\Db \AA$ reduces itself to a path $A \to X \to B[n]$ for some $n \leq 0$.  If we assume that $\AA$ is hereditary, then we have the following consequence.

\begin{theorem}\label{theorem:Paths}
Let $\AA$ be an indecomposable Ext-finite hereditary category, and let $A,B \in \Ob \AA$ be indecomposables.  There is an unoriented path of length at most two from $A$ to $B$.  If there is an oriented path from $A$ to $B$, then there is an oriented path from $A$ to $B$ of length at most two.
\end{theorem}

\begin{proof}
The statement is trivial if $\Hom(A,B) \not= 0$, $\Hom(B,A) \not= 0$, $\Ext(A,B) \not= 0$ and $\Ext(B,A) \not= 0$.  We will exclude these cases.

It is well-known that every object in the bounded derived category $\Db \AA$ of $\AA$ is isomorphic to a direct sum of its homologies.  In particular, we find that there cannot be a path from $A$ to $B[-n]$ for $n>0$ (see \cite{Ringel05}).

Since $\AA$ is indecomposable, we know (by \cite{Ringel05}) that there is a path from $A$ to $B[n]$ in $\Db \AA$ for a certain $n \gg 0$.  Using \cite{HappelZacharia08} this reduces to a path $A \to X[k] \to B[l]$ where $l \leq n$ and where $X \in \Ob \AA$ is indecomposable.  Since $\AA$ is hereditary, we find that $0 \leq k \leq 1$ and $0\leq l \leq k+1 \leq 2$.

The case where $k=l=0$ is trivial; a path of length two is then given by $A \to X \to B$.  If there is a path from $A$ to $B$, then we are in this case.

Assume now $l = 1$.  We will furthermore assume $k=0$, the case where $k=1$ is analogous.  In this case $\Hom(A,X) \not= 0$ and $\Ext(X,B) \not= 0$.  Let $0 \to B \to Y \to X \to 0$ be a nonsplit short exact sequence and note that $\Hom(B,Y') \not= 0$ for every direct summand $Y'$ of $Y$.  Applying the functor $\Hom(A,-)$, and using that $\Ext(A,B) = 0$ and $\Hom(A,X) \not= 0$, we find that $\Hom(A,Y) \not= 0$.  Thus there is an indecomposable direct summand $Y'$ of $Y$ and nonzero morphisms $A \to Y'$ and $B \to Y'$.  This gives an unoriented path between $A$ and $B$ of length two.

The last case we need to consider is where $l=2$ and thus $k=1$.  Since $\Ext(A,X) \not=0$ we may consider a nonsplit exact sequence $0 \to X \to Y \to A \to 0$.  Thus $\Hom(Y',A) \not= 0$ for every indecomposable direct summand $Y'$ of $Y$.  From $\Ext(X,B) \not= 0$, we obtain that $\Ext(Y,B) \not= 0$.  In particular, $\Ext(Y',B) \not= 0$ for some indecomposable direct summand $Y'$ of $Y$.  Consider a nonsplit exact sequence $0 \to B \to Z \to Y' \to 0$.  We have thus a path from $B$ to $A$ in $\AA$.  As before, this path can be shortened to a path of length at most two.
\end{proof}
\section{Some properties of uniserial categories}\label{section:Properties}

An Ext-finite abelian category $\AA$ is said to be \emph{uniserial} if, for every $X \in \ind \AA$, the subobjects of $X$ are linearly ordered.  This property is self-dual, thus the dual of an abelian hereditary uniserial category with Serre duality is again an abelian hereditary uniserial category with Serre duality.

An object $S \in \ind \AA$ is called \emph{peripheral} if there is an Auslander-Reiten sequence starting or ending at $S$ with an indecomposable middle term.  In particular, projective-injective objects are never peripheral.

In this section, we shall establish some facts about uniserial categories which we will need to complete the proof of the classification (Theorem \ref{theorem:Classification}).  Our main result in this section is Proposition \ref{proposition:FiniteShape} which states that every hereditary uniserial category with Serre duality is --in some way-- built up from length categories.  The proof of Proposition \ref{proposition:FiniteShape} goes in different steps.  We will start this section by proving that every indecomposable object $X$ has a simple top and simple socle, and then use these simple objects to construct (perpendicular) subcategories.

\begin{proposition}\label{proposition:Uniserial}
Let $\AA$ be an indecomposable Ext-finite abelian hereditary uniserial category with Serre duality.  Then the following hold:
\begin{enumerate}
\item All peripheral objects are simple.
\item Every object has a simple top and a simple socle.
\item The middle term of every Auslander-Reiten sequence has at most two direct summands.
\item If $X \to M_1 \oplus M_2$ is a left almost split map, then $X \to M_1$ is an epimorphism and $X \to M_2$ is a monomorphism, or vice versa.
\end{enumerate}
If $\AA$ is not semi-simple, then the simple objects are exactly the peripheral objects.
\end{proposition}

\begin{proof}
Let $X$ be an indecomposable object of $\AA$, and assume that $X$ is not projective.  We show that the middle term $M$ of an Auslander-Reiten sequence $0 \to \t X \to M \to X \to 0$ has at most two direct summands.

Assume that $M$ has at least two direct summands, $M_1$ and $M_2$.  We claim that the corresponding irreducible morphisms $\alpha_1: M_1 \to X$ and $\alpha_2: M_2 \to X$ may not be both epimorphisms or monomorphisms.

It follows directly from the definition of a uniserial object and irreducible morphisms that $\alpha_1$ and $\alpha_2$ may not both be monomorphisms.  If $\alpha_1$ and $\alpha_2$ are both epimorphisms, then the following commutative diagram
$$\xymatrix{
0 \ar[r]& \ker \alpha_1 \oplus \ker \alpha_2 \ar[r]\ar@{-->}[d] & M_1 \oplus M_2 \ar@{^{(}->}[d]\ar[r]& X \oplus X \ar[r] \ar[d] &0 \\
0 \ar[r]& \t X \ar[r] & M \ar[r]& X \ar[r] &0
}$$
shows that $\ker \alpha_1 \oplus \ker \alpha_2$ is a subobject of $\t X$.  Since $\t X$ is uniserial, every subobject must be indecomposable.  A contradiction.

The only possibility is thus that $X \to M_1$ is an epimorphism and $X \to M_2$ is a monomorphism, or vice versa.  In particular this implies that the middle term $M$ of the Auslander-Reiten sequence $0 \to \t X \to M \to X \to 0$ has at most two direct summands.

Let $f: X \to Y$ be an irreducible morphism between indecomposable objects.  If $f$ is an epimorphism then we claim that $\ker f$ is simple.  Dually, if $f$ is a monomorphism then $\coker f$ is simple.

Let us prove the claim in the case that $f$ is an epimorphism.  Let $K = \ker f$ and $S$ a subobject $K$.  We find following commutative diagram
$$\xymatrix{
0 \ar[r] & S \ar[r] \ar[d] & X \ar[r] \ar@{=}[d] & C \ar[r] \ar[d] & 0 \\
0 \ar[r] & K \ar[r]  & X \ar[r]  & Y \ar[r] & 0 }$$
where the rows are exact.  Since $X$ is uniserial, the quotient object $C$ is indecomposable, thus either $X \to C$ or $C \to Y$ is an isomorphism.  We conclude that either $S \cong K$ or $S \cong 0$, thus $K$ is simple.

Now we show that every indecomposable object has a simple top and a simple socle.  If $X$ is projective or injective, then Serre duality implies that $X$ has a simple top or a simple socle, respectively.  If $X$ is non-peripheral and non-projective, then we may consider the Auslander-Reiten sequence $0 \to \t X \to M \to X \to 0$ where $M$ is not indecomposable.  There are thus irreducible morphisms $\alpha_1: M_1 \to X$ and $\alpha_2: M_2 \to X$.  As shown before, one is a monomorphism of which the cokernel is the simple top of $X$.  The dual reasoning implies that if $X$ is non-peripheral and non-injective, then $X$ has a simple socle.

Finally let $X$ be a peripheral object.  Since $X$ is not projective-injective, $X$ is either the kernel of an irreducible epimorphism or the cokernel of an irreducible monomorphism and, as such, a simple object.  If $X$ is simple, then any irreducible $X \to Y$ or $Y \to X$ is a monomorphism (if $X$ is not injective) or an epimorphism (if $X$ is not projective), respectively (such irreducibles exist if $\AA$ is not semi-simple).  As shown before, the Auslander-Reiten sequence starting or ending in $X$ has an indecomposable middle term, hence $X$ is peripheral.
\end{proof}

\begin{corollary}\label{corollary:Connected}
Let $A, B \in \ind \AA$ be simple objects where $\AA$ is an indecomposable hereditary uniserial category with Serre duality.  There is an object $X \in \ind \AA$ which has $A$ as a subobject and $B$ as a quotient object, or vice versa.
\end{corollary}

\begin{proof}
Theorem \ref{theorem:Paths} says there is an (unoriented) path from $A$ to $B$ of length at most two.  If there is a path of length one, a path $A \to X \to B$, or a path $B \to X \to A$, then the statement has been shown.

If there is an unoriented path $A \to X \leftarrow B$ or $B \leftarrow X \rightarrow A$ then $A$ and $B$ are both either the socle or the top of $X$, and hence $A \cong B$.
\end{proof}

\begin{proposition}\label{proposition:MonoEpi}
Let $\AA$ be an abelian Ext-finite uniserial category, and let $X,Y \in \Ob \AA$ be indecomposables.  If there is a monomorphism $f_1: X \to Y$ and an epimorphism $f_2:X \to Y$ then $X \cong Y$.
\end{proposition}

\begin{proof}
Let $g,h: X \to Y$ be two maps which are not epimorphisms.  For all $\alpha, \beta \in k$ we have that $\im(\alpha g + \beta h) \subseteq \im g \cup \im h$.  Since $\AA$ is uniserial, this is different from $Y$.  We conclude that a linear combination of non-epimorphisms is not an epimorphism and hence the subset $V$ of $\Hom(X,Y)$ consisting of all morphisms which are not epimorphisms is a subspace.

Dually the subset $W$ of $\Hom(X,Y)$ consisting of all morphisms which are not monomorphisms is a subspace as well.

Since there is a monomorphism $f_1: X \to Y$ and an epimorphism $f_2:X \to Y$, neither $V$ nor $W$ are equal to $\Hom(X,Y)$.  We conclude that there is an $f \in \Hom(X,Y)$ which does not lie in $V \cup W$. This morphism in is an isomorphism $f:X \stackrel{\sim}{\rightarrow} Y$.
\end{proof}

The above proposition does not hold in general for Ext-finite abelian categories as the following example illustrates.

\begin{example}
We will assume the reader is familiar with (split) torsion theories.  Define a torsion theory on the category $\coh \bP^1$ of coherent sheaves on a projective line as follows: choose any closed point $P \in \bP^1$ and let $\TT$ (=class of torsion objects) be the full subcategory of $\coh \bP^1$ given by the sheaves $\GG$ with support in $P$.  We define $\FF$ (=class of torsionfree objects) as the full subcategory given by all coherent sheaves $\GG$ with $\Hom(k_P,\GG) = 0$, where $k_P$ is the simple sheaf supported at $P$.

We can consider the tilted category $\AA$ which has a natural torsion theory given by $(\FF,\TT[-1])$.  This category is hereditary (because $\Ext(\FF,\TT) = 0$) and Ext-finite.

In $\coh \bP^1$, every nonzero map in $\Hom(\OO,\OO(1))$ is a monomorphism; the cokernels are given by the simple torsion sheaves.  However a nonzero map in $\Hom_\AA(\OO,\OO(1))$ will be a monomorphism if and only if the cokernel in $\coh \bP^1$ does not lie in $\TT$ and will be an epimorphism otherwise (the kernel then lies in $\TT[-1] \subset \AA$).
\end{example}

\begin{proposition}\label{proposition:TopCollision}
Let $\AA$ be an abelian uniserial category.  Let $X_1,X_2 \in \ind \AA$ and let $Q$ be a nonzero quotient object (subobject) of both $X_1$ and $X_2$, then there is either an epimorphism (monomorphism) $X_1 \to X_2$ or an epimorphism (monomorphism) $X_2 \to X_1$.
\end{proposition}

\begin{proof}
We only prove the statement where $Q$ is a quotient object.  The case where $Q$ is a subobject is dual.  Let $K$ be the pullback of
$$\xymatrix{K \ar@{-->}[r] \ar@{-->}[d]& X_1 \ar[d] \\
X_2 \ar[r] & Q}$$
Since $X_2 \to Q$ is an epimorphism, so is the map $K \to X_1$.  In particular, the composition $K \to X_1 \to Q$ is an epimorphism.

Since $\AA$ is uniserial, the images of  $K' \to X_1$ are linearly ordered by inclusion (where $K'$ ranges over the direct summands $K$) and hence there is a direct summand $K'$ such that $K' \to X_1$ is an epimorphism.  Thus $K' \to X_1 \to Q = K' \to X_2 \to Q$ is also an epimorphism.  We show that $K' \to X_2$ is an epimorphism.

Let $C = \coker(K' \to X_2)$.  Since $\AA$ is uniserial and $C,Q$ are both quotient objects of $X_2$, we either have $X_2 \to Q = X_2 \to C \to Q$ or $X_2 \to C = X_2 \to Q \to C$.  In both cases, we find that the composition with $K' \to X_2$ gives zero.  Since $K' \to X_2 \to Q$ is an epimorphism, this is only possible in the latter case, if $C = 0$.  This shows that $K' \to X_2$ is an epimorphism as well.

Thus both $X_1$ and $X_2$ are quotient objects of $K'$.  Since $\AA$ is uniserial, this completes the proof.
\end{proof}

\begin{corollary}\label{corollary:QuotientsWithSameSocle}
Let $\AA$ be an abelian Ext-finite uniserial category.  Assume that there are two nonisomorphic indecomposable objects $X,Y \in \Ob \AA$ and a epimorphism $X \to Y$.  If there is an object $S \in \Ob \AA$ which is both a subobject of $X$ and a subobject of $Y$, then there is a monomorphism $Y \to X$.
\end{corollary}

\begin{proof}
By Proposition \ref{proposition:TopCollision} there is a either a monomorphism $X \to Y$ or a monomorphism $Y \to X$.  Proposition \ref{proposition:MonoEpi} excludes the former case.
\end{proof}

\begin{proposition}\label{proposition:LinearlyIndependent}
Let $\AA$ be a uniserial abelian category.  Let $X,Y \in \ind \AA$.  A set of nonzero morphisms $\{f_i: X \to Y\}$ with $\im f_i \not\cong \im f_j$ (as subobjects of $Y$) for different $i,j \in I$ is linearly independent.
\end{proposition}

\begin{proof}
Let $0 = \sum_{i \in J} a_i f_i$ be a nontrivial linear combination where $J$ is a finite subset of $I$ and all $a_i$'s are nonzero.  Since the morphisms are assumed to be nonzero, we know that $|J| \geq 2$.  Let $k,l \in J$ be such that $\im f_j \underset{\not=}{\subset} \im f_l \underset{\not=}{\subset} \im f_k$ for all $j \in J \setminus \{k,l\}$.

We know that $a_k f_k = -\sum_{i \not= k} a_i f_i$.  The right hand side of the equation factors through $\bigcup_{i \not= k} \im f_i = \im f_l \hookrightarrow X$, but the left hand side does not.  A contradiction.
\end{proof}

\begin{proposition}\label{proposition:FiniteLength}
Let $\AA$ be an abelian Ext-finite hereditary uniserial category with Serre duality.  Let $\SS \subseteq \Ob \AA$ be the set of simple objects in $\ind \AA$ and let $\SS_f \subseteq \SS$ be a finite subset.  The category $\AA_f = (\SS \setminus \SS_f)^\perp$ is an abelian hereditary uniserial length category with only finitely many simple objects.
\end{proposition}

\begin{proof}
Note that $\AA_f$ is an abelian hereditary uniserial category.  Let $X \in \Ob \AA$ be an object which is simple in $\AA_f$.  We know that $\soc_\AA(X) \in \SS_f$.  Proposition \ref{proposition:TopCollision} shows there can only be one simple object $X$ in $\AA_f$ with $\Hom_\AA(\soc_\AA(X),X) \not= 0$.  This shows there are at most finitely many simple objects in $\AA_f$: the number is bounded above by $|\SS_f|$.

We will proceed by showing that $\AA_f$ is a length category.  Seeking a contradiction, assume there is an indecomposable object $X \in \ind \AA_f$ with an infinite sequence of nonisomorphic quotient objects in $\AA_f$.  This implies that $X \in \ind \AA$ has an infinite sequence of nonisomorphic quotient objects in $\AA$ with socle in $\SS_f$.  Thus there is an $S \in \SS_f$ such that infinitely many quotient objects $\{X_i\}_i$ of $X$ have socle $S$.

If there is an infinite sequence $X_0 \twoheadrightarrow X_1 \twoheadrightarrow X_2 \twoheadrightarrow \cdots$, then Corollary \ref{corollary:QuotientsWithSameSocle} shows that there are monomorphisms $X_i \hookrightarrow X_0$.  We find an infinite set of morphisms $f_i:X_0 \twoheadrightarrow X_i \hookrightarrow X_0$ which are linearly independent (see Proposition \ref{proposition:LinearlyIndependent}).  Contradiction.

If there is an infinite sequence $\cdots X_2 \twoheadrightarrow X_1 \twoheadrightarrow X_0$ of quotient objects of $X$ with socle $S$, then there is an infinite sequence $\cdots \hookrightarrow K_2 \hookrightarrow K_1 \hookrightarrow K_0$ where $K_i = \ker(X \twoheadrightarrow X_i)$.  Note that the top of $K_i$ is given by $\t S$ where $S = \soc X$.  The dual of the previous argument shows that such a sequence cannot exist.
\end{proof}

The following lemma relates taking perpendicular categories in $\AA$ and perpendicular categories in length subcategories of $\AA$.

\begin{lemma}\label{lemma:MultiplePerpendiculars}
Let $\AA$ be an Ext-finite hereditary uniserial category with Serre duality.  Denote by $\SS \subseteq \Ob \AA$ the set of all simple objects in $\ind \AA$ and let $\SS_f \subseteq \SS$ be a finite subset.  Write $\AA_f = (\SS \setminus \SS_f)$.  For every $S \in \SS_f$, there is an object $X_S \in \Ob \AA_f$ such that $\AA_f \cap S^\perp = \AA_f \cap (X_S)^\perp$.
\end{lemma}

\begin{proof}
We will consider three cases.  In the first case, assume that $\Hom_\AA(S,Z)=0$ and $\Ext_\AA(S,Z)=0$, for all $Z \in \AA_f$.  In that case we can choose $X_S = 0$.

In the second case, assume that $\Hom_\AA(S,-)$ is nonzero on $\AA_f$.  Let $X_S$ be an object in $\AA_f$ with minimal length ($\AA_f$ is a length category by Proposition \ref{proposition:FiniteLength}) such that $\Hom(S,X_S) \not= 0$.  It follows from Proposition \ref{proposition:TopCollision} that $X_S$ is a simple object in $\AA_f$.  We want to show that $\AA_f \cap S^\perp = \AA_f \cap (X_S)^\perp$.  Without loss of generality, we may assume that $S \not\cong X_S$.

Proposition \ref{proposition:TopCollision} also implies that $\Hom(S,Z) \cong \Hom(X_S,Z)$ for all $Z \in \Ob \AA_f$.  Since $S$ is a subobject of $X_S$, we know that $\dim_k \Ext(S,Z) \leq \dim_k \Ext(X_S,Z)$, for all $Z \in \AA$ so that $\AA_f \cap S^\perp \supseteq \AA_f \cap (X_S)^\perp$.

For the other inclusion, we will show that $\Ext(X_S,Z) \not= 0$ implies $\Hom(S,Z) \not= 0$ or $\Ext(S,Z) \not= 0$.  Thus assume that $\Ext(X_S,Z) \not= 0$ and $\Ext(S,Z) = 0$.  In that case $\Ext(X_S/S, Z) \not= 0$ and this yields a commutative diagram
$$\xymatrix{&& 0 \ar[d] & 0 \ar[d] \\
&& Z \ar[d] \ar@{=}[r] & Z \ar[d] \\
0 \ar[r] & S \ar[r] \ar@{=}[d] & M \ar[r] \ar[d] & M' \ar[d] \ar[r] & 0 \\
0 \ar[r] & S \ar[r] & X_S \ar[r] \ar[d] & X_S/S \ar[d] \ar[r] & 0 \\
&& 0 & 0}$$
where the rows are exact, and the columns are exact and nonsplit.  Since $X_S$ is simple in $\AA_f$ and $Z \in \Ob \AA_f$, we see that $M$ is indecomposable and hence the rows are nonsplit as well.  Since $\AA$ is uniserial, the map $S \to M$ factors through $Z \to M$ and hence $\Hom(S,Z) \not= 0$ as required.

For the third case, we assume that $\Ext_\AA(S,-)$ is nonzero on $\AA_f$.  This shows that $S$ is not projective, and hence $\Ext_\AA(S,-) \cong \Hom_\AA(-,\t S)$.  This case is dual to the second case above.
\end{proof}

\begin{remark}
The previous lemma states that $S^\perp \cap \AA_f$ is the same as $X_S^\perp$ where the last perpendicular is taken in $\AA_f$.  In particular, Proposition \ref{proposition:GeigleLenzing} can be applied in this case.
\end{remark}

\begin{proposition}\label{proposition:FiniteShape}
Let $\AA$ be an Ext-finite hereditary uniserial category  with Serre duality.  Denote by $\SS \subseteq \Ob \AA$ the set of all simple objects in $\ind \AA$ and let $\SS_f$ be a finite subset.  The category $\AA_f = (\SS \setminus \SS_f)^\perp$ is equivalent to either
\begin{enumerate}
\item $\rep_k A_n$, for some $n \geq 0$ where $A_n$ has linear orientation, or
\item $\nilp_k \tilde{A}_n$, for some $n \geq 0$ where $\tilde{A}_n$ has cyclic orientation.
\end{enumerate}
If $\AA$ is directed then $\AA_f$ is of the first type.  If $\AA$ is not directed then $\AA_f$ is of the second type.
\end{proposition}

\begin{proof}
It follows from Proposition \ref{proposition:FiniteLength} that $\AA_f$ is a hereditary uniserial length category, and it follows from Corollary \ref{corollary:Connected} that $\AA_f$ is indecomposable.  We know from \cite{AmdalRingdal68} (see also \cite{ChenKrause09, CuadraGomezTorrecillas04, Gabriel73}) that $\AA_f \cong \rep_k A_n$ or $\AA_f \cong \nilp_k \tilde{A}_n$.

If $\AA$ is directed, then so is $\AA_f$ and hence $\AA_f \cong \rep_k A_n$.  Thus assume that $\AA$ is not directed.  There is a finite subset $\SS_{f'}$ of $\SS$, containing $\SS_f$ such that $\AA_{f'} = (\SS \setminus \SS_{f'})^\perp$ is not directed.  Hence $\AA_{f'}$ falls in the second category.  It follows from Lemma \ref{lemma:MultiplePerpendiculars} that $\AA_f$ also falls in the second category. 
\end{proof}

\section{Description and classification by 2-colimits}\label{section:2Colimit}

Let $\AA$ be a small hereditary uniserial category with Serre duality.  If $\AA$ is a length category, then the classification follows easily from \cite{AmdalRingdal68, ChenKrause09, CuadraGomezTorrecillas04, Gabriel73}.  Otherwise, we will use Proposition \ref{proposition:FiniteShape} to write $\AA$ as a filtered 2-colimit of such length subcategories.  We will need to consider two cases: one where $\AA$ is directed and one where $\AA$ is not directed.  We will concentrate on the latter, the former is similar (alternatively, one can use the classification of hereditary directed categories with Serre duality in \cite{vanRoosmalen06}).  We will thus assume that $\AA$ is not a length category and we want to show that $\AA$ is equivalent a big tube $\modc \widehat{k\LL^\bullet}$.

\subsection{An ordering on the set of simple objects}

In order to find the correct poset $\LL$, we will define an ordering on the set of simple objects of $\ind \AA$ (Definition \ref{definition:Ordering}).  Therefore, we will need the following proposition.  We will say an object $X \in \AA$ is \emph{endo-simple} when $\Hom(X,X) \cong k$.

\begin{proposition}\label{proposition:Ordering}
Let $\AA$ be an Ext-finite hereditary uniserial category with Serre duality.  Let $S,T_1,T_2$ be three simple objects such that there is a path from $S$ to both $T_1$ and $T_2$.  Then there are endo-simple objects $X_1$ and $X_2$, uniquely determined up to isomorphism, with socle $S$ and top $T_1$ and $T_2$, respectively.  Furthermore, there is a monomorphism $X_1 \to X_2$ or $X_2 \to X_1$.
\end{proposition}

\begin{proof}
It follows from Corollary \ref{corollary:Connected} that there is an object $Y_1$ with socle $S$ and top $T_1$.  Choose a nonzero endomorphism of $Y_1$ such that the image $X_1$ is minimal as a subobject of $Y_1$ (this is possible by Proposition \ref{proposition:LinearlyIndependent}).  It is clear that $X_1$ is endo-simple, has $S$ as socle, and has $T_1$ as top.  Likewise, one constructs $X_2$.

To show $X_1$ is unique up to isomorphism, let $X'_1$ be any endo-simple object with socle $S$ and top $T_1$.  By Propositions \ref{proposition:MonoEpi} and \ref{proposition:TopCollision}, we may assume that there is a monomorphism $X_1 \hookrightarrow X'_1$ and an epimorphism $X'_1 \twoheadrightarrow X_1$.  Since $X'_1$ is endo-simple, this shows that the the nonzero composition $X'_1 \twoheadrightarrow X_1 \hookrightarrow X'_1$ is an isomorphism and thus $X'_1 \cong X_1$.

The last statement follows from Proposition \ref{proposition:TopCollision}.
\end{proof}

\begin{definition}\label{definition:Ordering}
Let $\AA$ be an Ext-finite hereditary uniserial category with Serre duality.  Let $\SS \subseteq \ind \AA$ be the set of simple objects and let $S \in \SS$.  Denote by $\SS_S$ the subset of $\SS$ consisting of all simple objects $T$ such that there is a path from $S$ to $T$.  On $\SS_S$, we may define a partial ordering $\leq_S$ as follows.  For any $T_1,T_2 \in \SS_S$, let $X_1,X_2$ denote the corresponding objects as in Proposition \ref{proposition:Ordering}.  We then have
$$T_1 \leq_S T_2 \Leftrightarrow X_1 \subseteq X_2.$$
\end{definition}

\begin{remark}
It follows from Proposition \ref{proposition:Ordering} that $\SS_S,\leq_S$ is linearly ordered.  Furthermore, the object $S$ is always a minimal element in the ordering $\leq_S$.  If $\AA$ is not directed then the object $\t S \in \SS$ is a maximal element.
\end{remark}

\begin{remark}
If $\AA$ is not directed, then $\SS_S = \SS$.  If furthermore $\SS$ is an infinite set (thus $\AA$ is not a length category), then the category $\widehat{k\SS_S}$ is equivalent to a big loop.
\end{remark}

\subsection{The 2-functors $\aa$ and $\overline{\aa}$}

Let $\AA$ be a hereditary uniserial category with Serre duality.  Assume furthermore that $\AA$ is not directed.  Denote by $\SS \subseteq \Ob \AA$ the set of simple objects in $\ind \AA$.  If $\SS_i \subseteq \SS$ is a finite subset, then we define $\AA_i = (\SS \setminus \SS_i)^\perp \subseteq \AA$.

The poset category of finite subsets of $\SS$ is a filtered category which we denote by $\PP$; to ease notations, we will write $i \in \Ob \PP$ for $\SS_i \in \Ob \PP$.  We define a 2-functor $\aa:\PP \to \CAT$ as follows:
\begin{eqnarray*}
\aa(i) &=& \AA_i \\
\aa(s) &=& \AA_i \to \AA_j
\end{eqnarray*}
where $i \in \Ob \PP$ and $s:i \to j$ in $\Mor \PP$.  The map $\aa(s): \AA_i \to \AA_j$ is the canonical embedding.  It is clear that $2\colim \aa \cong \AA$.  We will denote the induced functors $\aa(i) \to \AA$ by $\sigma^\aa_i$.

We also have a 2-functor $\overline{\aa}$ given by $\overline{\aa}(i) = \Ind \AA_i$; the action on the maps is given by the standard lifting from $\aa(s): \AA_i \to \AA_j$ to $\Ind \AA_i \to \Ind \AA_j$ (see \S\ref{section:Grothendieck}).  Since the map $\aa(s): \AA_i \to \AA_j$ is fully faithful and exact, so is the map $\overline{\aa}(s): \Ind \AA_i \to \Ind \AA_j$.  We will denote the 2-colimit of $\oa$ by $\overline{\AA}$.  The induced functors $\overline{\aa}(i) \to \overline{\AA}$ will be denoted by $\sigma_i^{\oa}$.

It follows from Proposition \ref{proposition:AbelianColimit} that $\overline{\AA}$ is abelian.  We remark that $\overline{\AA}$ needs not to be a Grothendieck category as it is not necessarily closed under arbitrary direct sums.  In particular, $\overline{\AA}$ is not equivalent to $\Ind \AA$.

The obvious 2-natural transformation $\aa \to \overline{\aa}$ induces a fully faithful functor $\mu: \AA \to \overline{\AA}$.  We will use the following.

\begin{lemma}
The functor $\mu: \AA \to \overline{\AA}$ maps simple objects to simple objects, and every simple object of $\overline{\AA}$ lies in the essential image of $\mu$.
\end{lemma}

\begin{proof}
For every simple object $S$ in $\overline{\AA}$, there is an $i \in \PP$ and a simple object $S' \in \oa(i)$ such that $\sigma^{\aa}_i(S') \cong S$.  The simple objects of $\oa(i)$ lie in the essential image of $\aa(i) \to \oa(i)$.  The statement now follows easily.
\end{proof}

The following proposition will allow us to reduce some notations.

\begin{proposition}
There is a 2-natural equivalence $\oa \to \oa'$ where $\oa'(i) = (\SS \setminus \SS_i)^\perp \subset \overline{\AA}$.
\end{proposition}

\begin{proof}
There is an obvious 2-natural equivalence $\oa \to \oa'$ when restricting to the compact objects.  This lifts to a 2-natural transformation as required.
\end{proof}

In what follows, we may thus assume that $\oa$ satisfies the following properties.
\begin{enumerate}
\item For every two composable morphisms $s:i \to j$ and $t:j \to k$, we have $\oa(t \circ s) = \oa(t) \circ \oa(s)$.
\item For every morphism $s:i \to j$, we have that $\sigma^{\oa}_i = \sigma^{\oa}_j \circ \oa(s)$.
\end{enumerate}

\subsection{Injectives in $\overline{\AA}$}

The category of injectives of $\overline{\aa}(i) = \Ind \AA_i$ has been described in \S\ref{subsection:Tubes}.  We will now describe the category of injectives of $\overline{\AA}$.  First note that Proposition \ref{proposition:GeigleLenzing} shows that the embedding $\aa(s):\aa(i) \to \aa(j)$ has an exact left adjoint (as $\aa(i)$ is the perpendicular subcategory on finitely many exceptional simple objects in $\aa(j)$ by Lemma \ref{lemma:MultiplePerpendiculars}) and thus $\overline{\aa}(s):\overline{\aa}(i) \to \overline{\aa}(j)$ maps injective objects to injective objects by Proposition \ref{proposition:ToInjectives}.

We infer that $\sigma^{\oa}_i: \overline{\aa}(i) \to \overline{\AA}$ also maps injective objects to injective objects.  Indeed, if $I \in \oa(i)$ would be an injective object such that $\sigma^{\oa}_i(I) \in \overline{\AA}$ is not injective, then there is a nonsplit monomorphism $\sigma_i^{\oa}(I) \to X$ for some $X \in \overline{\AA}$.  This implies there is an $s: i \to j$ in $\PP$ and an object $X' \in \oa(j)$ such that $\oa(s)(I) \to X'$ is a nonsplit monomorphism.  Contradiction.

\begin{proposition}
The category $\overline{\AA}$ has enough injectives and the results of Theorem \ref{theorem:Matlis} holds.  The indecomposable injectives are given by the injective envelopes of the simple objects.
\end{proposition}

\begin{proof}
These properties are true for each category $\overline{\aa}(i)$ where $i \in \PP$ and they carry over through the 2-colimit.
\end{proof}

For each simple object $S \in \overline{\AA}$, we will denote its injective hull by $I(S)$.  Thus $I(S)$ is the unique (up to isomorphism) indecomposable injective in $\overline{\AA}$ such that there is a monomorphism $S \to I(S)$.  For every $S \in \SS$, we will fix such an object $I(S)$.

\begin{remark}
When the set $\SS$ is infinite, there is no object ``$\oplus_{S \in \SS} I(S)$'' in $\overline{\AA}$.
\end{remark}

\begin{remark}\label{remark:injectives}
Given two injective objects $I_1,I_2 \in \overline{\AA}$, there is a $j \in \PP$ such that
$$\Hom_{\overline{\AA}}(I_1, I_2) \cong \Hom_{\overline{\aa}(j)}(J_1,J_2)$$
where $\sigma^{\oa}_j J_1 \cong I_1$ and $\sigma^{\oa}_j J_2 \cong I_2$ (see Proposition \ref{proposition:FilteredColimit}).
\end{remark}

We will describe the Hom-spaces between indecomposable injectives in a similar way as in \cite{AmdalRingdal68} (see \S\ref{subsection:Tubes}).  We will fix an $S \in \SS$ and for every $T \in \SS$ we will write $A_T = \End_{\overline{\AA}} I(T)$.
\begin{enumerate}
	\item We know that $A_S \cong k[[x]]$ as algebras (see Remark \ref{remark:injectives}).  We will fix such an isomorphism and identify $A_S$ with $k[[x]]$ using this isomorphism.
	\item\label{enumerate:ST} For every $T \in \SS$, we have that $\Hom(I(S),I(T)) \cong k[[x]]$ as right $A_S$-modules (see Remark \ref{remark:injectives}).  Again, we will fix such an isomorphism and identify $\Hom(I(S),I(T))$ with $k[[x]]$.  To avoid confusion, we will sometimes write $1_{S,T}$ and $x_{S,T}$ for $1,x \in k[[x]] = \Hom(I(S),I(T))$.
	Note that there is a map $\varphi: \Hom(I(S),I(T)) \to A_S$ (which is an isomorphism of right $A_S$-modules) which is defined by
$$1_{S,T} \circ \varphi(f) = f.$$
	\item For every $T_1,T_2 \in \SS$, we define a map $\psi:\Hom(I(T_1),I(T_2)) \to \Hom(I(S),I(T_2))$ of left $A_{T_2}$-modules by
$$\psi(g) = g \circ 1_{S,T_1}.$$
Remark \ref{remark:injectives} implies $\psi$ will always be monomorphism and as such defines an isomorphism with a submodule.
If $T_1 \leq T_2$ (in the ordering given in Definition \ref{definition:Ordering}) then it follows from Remark \ref{remark:injectives} that every morphism in $\Hom(I(S),I(T_2))$ factors through $I(T_1)$ and thus through $1_{S,T_1}$ such that $\varphi$ is an isomorphism.
If $T_1 > T_2$ then no map in the image of $\psi$ will factor through $1_{S,T_1}$.  However, in this case Remark \ref{remark:injectives} shows that the image of $\psi$ is isomorphic to $xk[[x]]$ as a submodule of $\Hom(I(S),I(T_2))$.
We will use these isomorphisms to identify $\Hom(I(T_1),I(T_2))$ with $k[[x]]$ or $xk[[x]]$, respectively.  
\end{enumerate}

\begin{lemma}
With the above identifications, for all $T_1,T_2,T_3 \in \SS$ the composition
$$\Hom(I(T_2),I(T_3)) \otimes_k \Hom(I(T_1),I(T_2)) \to \Hom(I(T_1),I(T_3))$$
corresponds to the multiplication.
\end{lemma}

\begin{proof}
Consider the map $\theta = \varphi \circ \psi: \Hom(I(T_1),I(T_2)) \to A_S$ given by $f \circ 1_{S,T_1} = 1_{S,T_2} \circ \theta(f)$.  Using the above identifications, we have that $f = \varphi(f) \in k[[x]]$.  As a slight abuse of notation, we will denote by $\theta$ also corresponding maps $\Hom(T_2,T_3) \to A_S$ and $\Hom(T_1,T_3) \to A_S$.

To prove the statement of the lemma, one should prove that $\theta(g \circ f) = \theta(g) \circ \theta(f)$, for all $f \in \Hom(I(T_1),I(T_2))$ and $g \in \Hom(I(T_2),I(T_3))$ as composition is multiplication in $A_S$.  Since $1_{S,T_1}$ is a monomorphism in the category of injectives of $\overline{\AA}$, it suffices to show that $1_{S,T_1} \circ \theta(g \circ f) = 1_{S,T_1} \circ \theta(g) \circ \theta(f)$.  we have
\begin{eqnarray*}
1_{S,T_3} \circ \theta (g \circ f) &=& (g \circ f) \circ 1_{S,T_1} \\
&=& g \circ (f \circ 1_{S,T_1}) \\
&=& g \circ 1_{S,T_2} \circ \theta(f) \\
&=& 1_{S,T_3} \circ \theta(g) \circ \theta(f)
\end{eqnarray*}
which yields that $\theta(g \circ f) = \theta(g) \circ \theta(f)$.
\end{proof}

Recall that, for each $i \in \PP$, the category $\oa(i)$ is the subcategory $(\SS \setminus \SS_i)^\perp \subset \overline{\AA}$.  For every $i \in \PP$, we will choose a basic injective cogenerator $\Ja_i \cong \bigoplus_{I(S) \in \oa(i)} I(S) \in \oa(i)$.  We have that $\Hom(\Ja_i,\Ja_i) \cong \oplus_{j,k} \Hom(I_j,I_k)$ and this last one can be written as
$$\begin{pmatrix}
k[[x]] & xk[[x]] & \cdots & xk[[x]] & xk[[x]] \\
k[[x]] & k[[x]] & \cdots & xk[[x]] & xk[[x]]\\
\vdots & & \ddots & \vdots & \vdots \\
k[[x]] & k[[x]] &\cdots & k[[x]] &xk [[x]] \\
k[[x]] & k[[x]] &\cdots & k[[x]] &k [[x]]
\end{pmatrix}$$
using the choices made above.  The algebra $\Hom(\Ja_i,\Ja_i)$ has the structure of a pseudocompact algebra when endowed with the usual topology, i.e. generated by the kernels of the maps $\Hom(\Ja_i,\Ja_i) \to \Hom(X,\Ja_i)$ where $X$ ranges over the finite length subobjects of $\Ja_i$.  Furthermore, this is the unique pseudocompact topology (see \S\ref{subsection:Tubes}).  For any arrow $s:i \to j$ in $\PP$, the algebra map $\Hom(\Ja_i,\Ja_i) \to \Hom(\Ja_j,\Ja_j)$ induced by $\oa(s)$ is continuous.

\subsection{The 2-functors $\bb$ and $\overline{\bb}$}

We had assumed that $\AA$ is not directed and has infinitely many nonisomorphic simple objects.  Choose an $S \in \SS$ and write $\LL = \SS_S,\leq_S$, the poset from Definition \ref{definition:Ordering}.  We define the big tube $\BB$ as $\modc \widehat{k\LL^\bullet}$ and identify the set of isomorphism classes of simple objects of $\BB$ with $\SS$ in the obvious way.

Similarly as with $\AA$, we will define $\BB_i = (\SS \setminus \SS_i)^\perp \subseteq \BB$, and 2-functors $\bb, \overline{\bb}:\PP \to \CAT$ by $\bb(i) = \BB_i$, and $\overline{\bb}(i) = \Ind \BB_i$.  We will write $2\colim \overline{\bb} = \overline{\BB}$ and identify $\ob(i)$ with $(\SS \setminus \SS_i)^\perp \subseteq \overline{\BB}$.  Also, for every object $S \in \SS$ there is a corresponding indecomposable injective $I(S) \in \overline{\BB}$.

As above, we will identify $\Hom(T_1,T_2)$ with $k[[x]]$ (if $T_1 \leq T_2$) or $xk[[x]]$ (if $T_1 > T_2$) such that composition corresponds with multiplication.  For each $i \in \PP$, we will choose an injective cogenerator $\Jb_i \cong \oplus_{I(S) \in \ob(i)} I(S)$ of $\ob(i)$.

We wish to show that $\AA \cong \BB$ by showing that $\aa \cong \bb$.  The last 2-natural equivalence will be induced by a 2-natural equivalence $\rho:\oa \to \ob$ given by functors $(\rho_i: \oa(i) \to \ob(i))_{i \in \PP}$ which map $\Ja_i$ to $\Jb_i$ and by the mapping
\begin{eqnarray*}
\Hom(\Ja_i,\Ja_i) &\to& \Hom(\Jb_i, \Jb_i) \\
f = \sum_{S,S' \in \SS_i} e_S f e_{S'} &\mapsto& \sum_{S,S' \in \SS_i} e_S f e_{S'}
\end{eqnarray*}
where $e_S$ is the idempotent corresponding with the direct summand $I(S)$ of $\Ja_i$ or $\Jb_i$.  Recall that for every $S,S' \in \SS$ we have that both $e_S \Hom(\Ja_i,\Ja_i) e_{S'}$ and $e_S \Hom(\Jb_i,\Jb_i) e_{S'}$ have been identified with $k[[x]]$ or $xk[[x]]$ so that the above map is indeed well-defined .  It follows from Corollary \ref{corollary:WhenEquivalent} that the functors $\rho_i$ are well-defined (they are left exact functors commuting with direct sums).  Each $\rho_i$ has an obvious quasi-inverse and hence they are equivalences.

\begin{lemma}\label{lemma:oaob}
The functors $(\rho_i: \oa(i) \to \ob(i))_{i \in \PP}$ defined above give a 2-natural equivalence $\oa \to \ob$.
\end{lemma}

\begin{proof}
We need only to show that that $\ob(s) \circ \rho_i:\oa(i) \to \ob(j)$ and $\rho_j \circ \oa(s):\oa(i) \to \ob(j)$ are naturally equivalent for any $s:i \to j$ in $\Mor \PP$.

There is an idempotent $e = \sum_{I(S) \in \oa(i)} e_{I(S)} \in \Hom(\Ja_j,\Ja_j)$ such that $\oa(s)(\Ja_i) \cong \ker e$.  It follows from the definition of $\rho_j$ that the diagram
$$\xymatrix{0 \ar[r] & \rho_j \circ \oa(s) \Ja_i \ar[r] & \rho_j \Ja_j \ar[r]^{\rho_j e}\ar@{=}[d] & \rho_j \Ja_j\ar@{=}[d] \\
0 \ar[r] & \ob(s) \circ \rho_i \Ja_i \ar[r] & \Jb_j \ar[r]^e & \Jb_j}$$
commutes (the rows are exact).  This shows that $\rho_j \circ \oa(s) \Ja_i \cong \ob(s) \circ \rho_i \Ja_i$ and thus by Corollary \ref{corollary:WhenEquivalent}, so that $\rho_j \circ \oa(s) \cong \ob(s) \circ \rho_i$.
\end{proof}

\begin{proposition}\label{proposition:AAisBigTube}
The categories $\AA$ and $\BB$ are equivalent.
\end{proposition}

\begin{proof}
It follows from Lemma \ref{lemma:oaob} that $\oa$ and $\ob$ are 2-naturally equivelent.  Restricting to the compact objects, this gives a 2-natural equivalence $\aa \to \bb$.  Hence $\AA \cong \BB$.
\end{proof}

\subsection{Classification} Let $\AA$ an indecomposable hereditary abelian Ext-finite uniserial category with Serre duality.  We have shown above that $\AA$ is equivalent to a big tube when $\AA$ is not directed.  To complete the classification, it suffices to consider the case where the category $\AA$ is directed.  This case can be handled in a similar way: let $\SS \subseteq \ind \AA$ be the set of simple objects, then $\SS$ has a linear ordering given by $S \leq T$ if there is a path from $S$ to $T$.

We can then construct 2-functors $\aa$, $\oa$ as above (note that $\aa(s):\aa(i) \to \aa(j)$ will map injective objects to injective objects).  The category $\overline{\AA} = 2\colim \oa$ has enough injectives, and we can find an identification
$$\Hom(I(S),I(T)) = \left\{ \begin{array}{ll} k & S \leq T \\ 0 & S > T \end{array} \right.$$
such that composition corresponds to multiplication.

We then define $\BB = \repc \LL$ where $\LL = \SS$ and the 2-functors $\bb, \ob$, and show that $\overline{\BB} = 2\colim \ob$ as above.  We can then show that $\oa \cong \ob$ so that $\AA \cong \BB$ as required.  We leave the details to the reader.

\begin{theorem}\label{theorem:Classification}
Let $\AA$ be an essentially small $k$-linear abelian Ext-finite uniserial hereditary category with Serre duality.  Then $\AA$ is equivalent to one of the following
\begin{enumerate}
\item the category $\repc \LL$ where $\LL$ is a locally discrete linearly ordered poset, either without minimal or maximal elements, or with both a minimal and a maximal element,
\item a (big) tube.
\end{enumerate}
\end{theorem}

\begin{proof}
If $\AA$ is a length category, then it follows from \cite{AmdalRingdal68}, \cite{Gabriel73}, \cite{CuadraGomezTorrecillas04} or \cite{ChenKrause09} that $\AA$ is equivalent to one of the above categories.

If $\AA$ is not a length category and not directed, then it follows from Proposition \ref{proposition:AAisBigTube} that $\AA$ is equivalent to the (big) tube $\BB$.  If $\AA$ is not a length category but is directed, then it follows in a similar way that $\BB \cong \repc \LL$ for a linearly ordered poset as in the statement of the theorem.
\end{proof}

\providecommand{\bysame}{\leavevmode\hbox to3em{\hrulefill}\thinspace}
\providecommand{\MR}{\relax\ifhmode\unskip\space\fi MR }
% \MRhref is called by the amsart/book/proc definition of \MR.
\providecommand{\MRhref}[2]{%
  \href{http://www.ams.org/mathscinet-getitem?mr=#1}{#2}
}
\providecommand{\href}[2]{#2}

\end{document}